\documentclass[11pt]{article}
\usepackage[utf8]{inputenc}
\usepackage[T1]{fontenc}
\usepackage{comment}
\usepackage{authblk}
\usepackage{relsize}
\usepackage{changepage}
\usepackage{mathtools}
\usepackage[english]{babel}
\usepackage{amsmath}
\usepackage{amsthm}
\usepackage{amsfonts}
\usepackage{amssymb}
\usepackage{graphicx}
\usepackage[usenames,dvipsnames]{color}
\theoremstyle{plain}
\newtheorem{theorem}{Theorem}[section]
\newtheorem{corollary}[theorem]{Corollary}
\newtheorem{proposition}[theorem]{Proposition}
\newtheorem{lemma}[theorem]{Lemma}

\newtheorem{claim}[theorem]{Claim}

\makeatletter
\newcommand{\vast}{\bBigg@{4}}
\newcommand{\Vast}{\bBigg@{5}}
\makeatother
\definecolor{bulgarianrose}{rgb}{0.28, 0.02, 0.03}
\definecolor{gray}{rgb}{0.5, 0.5, 0.5}

\theoremstyle{definition}
\newtheorem{remark}[theorem]{Remark}

\usepackage[left=2cm,right=2cm,top=2cm,bottom=2cm]{geometry}
\usepackage{amssymb}
\usepackage{hyperref}
\makeatletter
\def\namedlabel#1#2{\begingroup
    #2%
    \def\@currentlabel{#2}%
    \phantomsection\label{#1}\endgroup
}
\usepackage[normalem]{ulem}
\usepackage{dsfont}
\usepackage{accents}
\usepackage{verbatim}
\usepackage{ulem}
\usepackage{pgf,tikz,pgfplots}
\pgfplotsset{compat = 1.16}
\usepackage[all]{xy}

\newcommand{\cD}{\mathcal{D}}
\renewcommand{\Pr}{\mathbb{P}}
\newcommand{\cC}{\mathcal{C}}

\newcommand{\cA}{\mathcal{A}}
\newcommand{\cB}{\mathcal{B}}
\newcommand{\cE}{\mathcal{E}}
\newcommand{\cF}{\mathcal{F}}
\newcommand{\cG}{\mathcal{G}}
\newcommand{\cH}{\mathcal{H}}
\newcommand{\bP}{\mathbb{P}}
\newcommand{\bE}{\mathbb{E}}

\newcommand{\eps}{\varepsilon}

\usepackage{subfig}

\usepackage{enumerate}

\title{\scshape
  Percolation on dense random graphs with given degrees}

\author[1]{Lyuben Lichev}
\author[1,3]{Dieter Mitsche\footnote{DM has been supported by Fondecyt grant 1220174 and by GrHyDy ANR-20-CE40-0002.}}
\author[2,4]{Guillem Perarnau\footnote{GP was supported by the Spanish Agencia Estatal de Investigaci\'on under projects PID2020-113082GB-I00 and the Severo Ochoa and Mar\'ia de Maeztu Program for Centers and Units of Excellence in R\&{}D (CEX2020-001084-M). }}
\affil[1]{Univ. Jean Monnet and Institut Camille Jordan, Saint-Etienne, France}
\affil[2]{IMTECH and Departament de Matemàtiques, Universitat Polit\`ecnica de Catalunya, Spain}
\affil[3]{IMC, Pontif\'{i}cia Univ. Cat\'{o}lica, Chile}
\affil[4]{Centre de Recerca Matem\`atica, Bellaterra, Spain}

\begin{document}

\maketitle

\begin{abstract}
In this paper, we study the order of the largest connected component of a random graph having two sources of randomness: first, the graph is chosen randomly from all graphs with a given degree sequence, and then bond percolation is applied. Far from being able to classify all such degree sequences, we exhibit several new threshold phenomena for the order of the largest component in terms of both sources of randomness. We also provide an example of a degree sequence for which the order of the largest component undergoes an unbounded number of jumps in terms of the percolation parameter, giving rise to a behavior that cannot be observed without percolation.
\end{abstract}

\section{Introduction}

Given $n\in \mathbb{N}$, a \emph{degree sequence $\cD_n = (d^{(n)}_1,\dots, d^{(n)}_n)$} is a sequence of non-negative integers such that $m(n) = \sum_{i=1}^n d^{(n)}_i$ is even. 
A degree sequence is \emph{feasible} if there exists a simple graph $G$ (that is, a graph containing no loops or multiple edges) with vertex set $[n]:=\{1,\dots,n\}$ such that $i$ has degree $d^{(n)}_i$ for all $i\in [n]$. Given a sequence of feasible degree sequences $(\cD_n)=(\cD_n)_{n\geq 1}$, we consider the sequence of random graphs $(\mathbf{G}(\cD_n))_{n\ge 1}$ where $\mathbf{G}(\cD_n)$ is chosen uniformly at random among the all simple graphs on $[n]$ where vertex $i$ has degree $d_i^{(n)}$ for all $i\in [n]$.

Random graphs with a given degree sequence is among the central topics of modern random graph theory. The influential work of Molloy and Reed~\cite{molloy1995critical} set a starting point for understanding the component structure of $\mathbf{G}(\cD_n)$: in particular, they introduced a criterion for the existence of the so-called \emph{giant component}, that is, a component of order linear in $n$, under some technical conditions on $\cD_n$. The range of applicability of this criterion has been extended in numerous subsequent works, also providing a more detailed description of the components~\cite{bollobas2015old, hatami2012scaling, janson2009new, joos2016how, joseph2010component, molloy1998size}. 

Given $p\in [0,1]$ and a graph $G$, we denote by $G_p$ the random subgraph of $G$ in which each edge of $G$ appears in $G_p$ independently (from the other edges) with probability $p$. We call $G_p$ the \emph{$p$-(bond)-percolation} of $G$. 
In this paper we will focus on the component structure of $\mathbf{G}(\cD_n)_p$. In this setting there are two levels of randomness: first, we choose a random graph $\mathbf{G}(\cD_n)$, and then we percolate it with probability $p$. We remark that the distribution of $\mathbf{G}(\cD_n)_p$ is \emph{not} necessarily uniform among all graphs with the same degree sequence.\footnote{For example, consider the set of graphs on four vertices and degree sequence $(3,3,2,2)$. There is a unique such graph satisfying that the two vertices of degree 2 do not form an edge. After percolation one obtains a 1--regular graph with some non-zero probability depending on $p$ but only two of the three possible perfect matchings appear with non-zero probability.}

The component structure of $\mathbf{G}(\cD_n)_p$ is well understood for regular degree sequences, that is, $d_i^{(n)}=d$ for all $i\in [n]$ and some integer $d\geq 3$. Goerdt~\cite{goerdt2001giant} proved the existence of a critical probability $p_c(d):=1/(d-1)$ such that the existence of a giant component undergoes a sharp threshold phenomenon at $p_c(d)$. While his result only applies for fixed $d$, the same conclusion holds for any $3\leq d=d(n)\leq n-1$; see also~\cite{joos2018critical} where the number of vertices in the largest component within the critical window was also shown to be typically of order $\Theta(n^{2/3})$. 
Bond percolation at criticality has been studied in more detail~\cite{joos2018critical,nachmias2010critical, pittel2008edge, riordan2012phase}.
Fountoulakis~\cite{fountoulakis2007percolation} and Janson~\cite{janson2008percolation} independently initiated the analysis of $\mathbf{G}(\cD_n)_p$, proving that 
$$
p_c=p_c(\cD_n) := \frac{\sum_{i=1}^n d_i^{(n)}}{\sum_{i=1}^n d_i^{(n)}(d_i^{(n)}-1)}
$$
is the threshold probability for the existence of a giant component under some conditions such as bounded maximum degree or bounded degree variance. 
These results were extended by Bollob\'as and Riordan~\cite{bollobas2015old} who proved the existence of a threshold provided that the degree distribution weakly converges (meaning that the proportion of vertices of degree $k$ rescaled by $n$ has a limit for every $k\ge 1$) and is uniformly integrable (that is, the average degree has a limit). These conditions in particular imply that the random graph is \emph{sparse}, that is, $m(n)= O(n)$. 
The approach of Bollob\'as and Riordan analyzes the \emph{configuration model} where a predetermined number of half-edges is attached to every vertex and these half-edges are then matched uniformly at random, possibly leading to loops or multiple edges.
An essential fact used in their proof is that the configuration model produces a simple graph with small though still sufficiently large probability when the average degree grows slowly to infinity. This can be widely used to analyze random graphs with given degrees but unfortunately not in our more general setting.
Fountoulakis, Joos and the third author~\cite{fountoulakis2022percolation} studied bond percolation when the convergence and integrability conditions are dropped but the degree sequence remains sparse. The random percolated graph at criticality has been recently studied by Dhara, van der Hofstad and van Leeuwarden in~\cite{dhara2021critical}.

Bollob\'as and Riordan noted the interest for the \emph{dense} case, characterized by the convergence of $m(n)/n$ to infinity (see Remark 24 in~\cite{bollobas2015old}). 
In the configuration model, their result can be applied to some almost regular dense degree sequences to show that $p_c$, which now tends to $0$ as $n\to\infty$, is a threshold for the existence of a giant component.
However, $p_c$ is not always the threshold in the dense case. 
For instance, there exist dense degree sequences for which the threshold probability is bounded away from zero as $n\to \infty$~(see Section 12 in~\cite{fountoulakis2022percolation}). Therefore, a more general approach is needed to fully understand the evolution of $\mathbf{G}(\cD_n)_p$ for dense sequences.

In this paper we provide a number of results on the sudden appearance of a large (but not necessarily linear order) component, shedding light on the evolution of percolated dense random graphs. The spirit behind most theorems below is guided by the following well-known intuition: if there exists a set of vertices having in expectation large degree after percolation, this set, together with components attached to it, will result in a large component. However, the source of randomness coming from percolation yields new phenomena. 
Firstly and contrarily to the sparse case, percolated dense random graphs can exhibit arbitrarily many linear jumps in the size of their largest component (see  Section~\ref{sec:exm1}); Theorems~\ref{thm 1} and~\ref{thm 2}, 
and Corollary~\ref{thm 3} establish different criteria to detect such jumps, while Theorem~\ref{thm 4} provides a tool to show that no such jump exists in a given probability interval. 
Secondly and most surprisingly, they might experience arbitrarily many jumps in the asymptotic order of their largest component (see Theorem~\ref{thm example}); 
Theorems~\ref{thm 6} and~\ref{thm 7} establish criteria to determine the size of sublinear largest components. Finally, the approach taken to study percolation on dense random graphs is completely different to the usual techniques for sparse random graphs and has a more combinatorial flavor.

\paragraph{Notation and terminology.} Given $f,g: \mathbb N\to \mathbb N$, we say that $f(n)\ll g(n)$ (or $f(n) = o(g(n))$, or $g(n)\gg f(n)$) if, for any $\eps>0$, there exists $n_\eps\in \mathbb N$ such that $f(n)\leq \eps g(n)$ for all $n\geq n_\eps$.

In this paper all graphs are simple (that is, without loops and multiple edges) unless explicitly stated otherwise. For a graph $G = (V,E)$, we call $|V|$ the \emph{order} of $G$ and we call $|E|$ the \emph{size} of $G$. For the sake of convenience, below we will sometimes talk about oriented edges $uv$, that is, we assume that the edge is oriented from $u$ to $v$ (although our graphs remain undirected). We also denote by $L_1(G)$ the order of a largest component in $G$. For a vertex $v\in V$, we denote by $d_G(v)$, or simply $d(v)$, the degree of $v$ in $G$.

Recall that $m = m(n) = \sum d^{(n)}_i$. As already pointed out, we are particularly interested in the case when the average degree $m/n$ tends to infinity as $n\to \infty$. 
We adopt the convention that random graphs with prescribed degrees are represented by bold symbols like $\mathbf{G}(\cD_n)$, or sometimes simply $\mathbf{G}$ or $\mathbf{H}$, to distinguish them from general graphs.
For a graph $G$ of order $n$ and $d\in\{0,\dots, n-1\}$, let $S_G(d)$ be the set of vertices of $G$ of degree at least $d$. From this point on, we abuse notation and write $\cD_n = (d_1, \dots, d_n)$ instead of $\cD_n = (d^{(n)}_1, \dots, d^{(n)}_n)$, and $S_n(d)$ instead of $S_{\mathbf{G}(\cD_n)}(d)$. 
Moreover, we assume throughout the paper that for all $i\in [n]$, $d_i\ge 1$ since isolated vertices in $\mathbf{G}(\cD_n)$ are irrelevant for our analysis. For a set $S\subseteq V(\mathbf{G}(\cD_n))$, we denote by $\Delta(S)$ the maximum degree (in $\mathbf{G}(\cD_n)$) of a vertex in $S$, and we write $d(S) = \sum_{v\in S} d(v)$. Moreover, for two disjoint sets $U, W\subseteq V(\mathbf{G}(\cD_n))$, we denote by $e(U,W)$ the number of edges between $U$ and $W$.

For a sequence of probability spaces $(\Omega_n, \mathcal F_n, \mathbb P_n)_{n\geq 1}$ and a sequence of events $(A_n)_{n\geq 1}$ with $A_n\in \mathcal F_n$ for every $n\geq 1$, we say that $(A_n)_{n\geq 1}$ happens \textit{asymptotically almost surely} (\textit{a.a.s.}) if $\lim_{n\to \infty}\mathbb P_n(A_n) = 1$.\par

To simplify the notation, we omit the use of $\lfloor \cdot \rfloor$ and $\lceil \cdot \rceil$ whenever it is clear from the context and it does not affect the proofs.

\subsection{Our results.}

For $\delta\in [0,1]$, a graph is called \emph{$\delta$-large} if it contains a component of order at least $\delta n$. 
Our first theorem is a criterion for the existence of a $\delta$-large component in $\mathbf{G}(\cD_n)$. 
While it does not provide a sharp threshold, it widely generalizes the existence of a threshold phenomenon for dense random regular graphs. 
Its statement corroborates the following intuition: if there are $\delta n$ vertices with expected degree $\gg 1$ after percolation, then a.a.s.\ almost all of them must participate in the same component of order at least $\delta n + o(n)$.

\begin{theorem}[Percolation threshold for $\delta$-large]\label{thm 1}
Let $\cD_n$ be a degree sequence such that there exist $\delta\ge 0$ and $d = d(n)$ satisfying $|S_n(d)|= \delta n +o(n)$. We have
\begin{enumerate}
    \item [(a)] \label{item1} if $p \ll 1/d$, then
    \begin{align*}
        \bP[L_1(\mathbf{G}(\cD_n)_p)\le \delta n+o(n)]=1-o(1).
    \end{align*}
    \item [(b)] \label{item2} if $p\gg 1/d$, then
    \begin{align*}
        \bP[L_1(\mathbf{G}(\cD_n)_p)\ge \delta n - o(n)]=1-o(1).
    \end{align*}
Moreover, a $(1-o(1))$--proportion of the vertices in $S_n(d)$ belong to the same connected component a.a.s.
\end{enumerate} 
\end{theorem}

\begin{remark}
Theorem~\ref{thm 1} can be used to describe the behavior of $\mathbf{G}(\cD_n)_p$ for $p=\Theta(d^{-1})$ under alternative conditions. For instance, if there exists $d'\gg d$ with $|S_n(d')|\geq \delta'n -o(n)$, then by applying (b) to $d'$, we obtain that for any $p = \Theta(d^{-1})$,
    \begin{align*}
        \bP[L_1(\mathbf{G}(\cD_n)_p)\ge \delta'n-o(n)]=1-o(1).
    \end{align*}
Similarly, if there exists $d''\ll d$ with $|S_n(d'')|\leq \delta''n -o(n)$, by applying (a) to $d''$, we obtain that for any $p = \Theta(d^{-1})$,
    \begin{align*}
        \bP[L_1(\mathbf{G}(\cD_n)_p)\le \delta''n-o(n)]=1-o(1).
    \end{align*}
\end{remark}

\begin{remark}
In a general setting, Theorem~\ref{thm 1} is as precise as it can get (see Section~\ref{sec:exm1}). 
In particular, if there exists a function $\omega$ such that $\lim_{n\to \infty} \omega(n)\to \infty$ and $d = d(n)$ satisfying $|S_n(d)| = \delta n + o(n)$ and $|S_n(\omega^{-1} d)\setminus S_n(\omega d)| = o(n)$, then a.a.s.\ $L_1(\mathbf{G}(\cD_n)_p) = \delta n - o(n)$. 
The linearity of $|S_n(d)|$ is also necessary for the conclusion in Part~(b). 
Intuitively, if $|S_n(d)| = o(n)$, then vertices of high degree could typically connect to lower degree vertices and thus would fail to form a dense component between themselves.
To be more precise, consider the following example.
Define a function $\omega = \omega(n)$ slowly going to infinity,~and consider a degree sequence with $|S_n(d)| = n/\omega^5$ vertices of degree $d\gg 1$ and $n-n/\omega^5$ vertices of degree $d' = d/\omega^3$.
Also, set $p = \omega/d$. Then, an easy application of the switching method allows us to conclude that a typical vertex of degree $d$ has $(1+o(1)) d\cdot (d|S_n(d)|/d'n)\approx d/\omega^2\ll 1/p$ neighbors in $S_n(d)$ before percolation. Thus, after percolation, a vertex in $S_n(d)$ typically has no neighbors in $S_n(d)$ and $(1+o(1)) pd = \omega+o(\omega)$ neighbors of degree $d'$.
At the same time, a set of $\omega+o(\omega)$ vertices of degree $d'$ has $(\omega+o(\omega)) pd' = o(1)$ expected neighbors after percolation.
This shows that $\mathbf{G}(\cD_n)_p$ consists of small connected components mostly given by stars of order $\omega + o(\omega)$ with centers in $S_n(d)$ and isolated vertices, and a comparison with a two-type branching process similar to the one from the proof of Theorem~\ref{thm example} is sufficient to show that a.a.s.\ all components have order $o(|S_n(d)|)$.
\end{remark}

\begin{remark}\label{rem:1.4}
It may be deduced from the proof of Theorem~\ref{thm 1} that the error bounds in both (a) and (b) are $\exp(-\omega(n)n)$ where $\omega$ is a function grows to infinity at speed $o(1/(pd))$ in (a) and $o(pd)$ in (b).
Moreover, the following version of Theorem~\ref{thm 1} may be shown with very minor modifications of our proof strategy: for every $\eps \in (0,1)$, there are constants $c_1 = c_1(\delta, \eps) > 0$ and $c_2 = c_2(\delta, \eps) > 0$ such that, given $p\ge \min(1, c_1/d)$, with probability at least $1-\exp(-c_2 n)$,
there is a connected component in $\mathbf{G}(\cD_n)_p$ containing an $(1-\eps)$-proportion of the vertices in $S_n(d)$.
\end{remark}

The next three theorems give a finer description of the threshold under stronger conditions focusing on the existence of \emph{jumps} in the order of the largest component (i.e., when small increments in the percolation probability give rise to an increase of the order of the largest component by a factor $c > 1$).

The first of these three results provides a sufficient condition for non-existence of jumps. Recall that $p_1$-percolation and $p_2$-percolation with $p_1 < p_2$ may be naturally coupled via a percolation process on two independent stages with parameters $p_1$ and $(p_2-p_1)/(1-p_1)$.

\begin{theorem}[Sublinear size: no jump]\label{thm 4}
Let $\cD_n$ be a degree sequence such that there exist $d_1 = d_1(n)$, $d_2 =  d_2(n) \leq d_1$ and $\omega=\omega(n)\to \infty$ as $n\to \infty$ satisfying
\begin{equation*}
    |S_n(\omega^{-1}d_2)\setminus S_n(\omega d_1)|= o(n).
\end{equation*}
Let $p_1 = \dfrac{1}{d_1}$ and $p_2 = \dfrac{1}{d_2}$. Then, 
\begin{equation*}
    \mathbb P\left[L_1(\mathbf{G}(\cD_n)_{p_2}) - L_1(\mathbf{G}(\cD_n)_{p_1}) = o(n)\right] = 1-o(1).
\end{equation*}
\end{theorem}

\begin{remark}
By taking $d_1'=\omega d_1$ and $d_2'=\omega^{-1}d_2$, we can rephrase Theorem~\ref{thm 4} as follows: if $$|S_n(d_2')\setminus S_n(d_1')|=o(n),$$ 
then for any $p_1,p_2$ satisfying $1/d_1'\ll p_1\leq p_2\ll 1/d_2'$, we have
\begin{equation*}
    \mathbb P\left[L_1(\mathbf{G}(\cD_n)_{p_2}) - L_1(\mathbf{G}(\cD_n)_{p_1}) = o(n)\right] = 1-o(1).
\end{equation*}
This form of Theorem~\ref{thm 4} is more suitable for comparison with the following Theorem~\ref{thm 2} and Corollary~\ref{thm 3}.
\end{remark}

The next two theorems, contrary to Theorem~\ref{thm 4}, provide a sufficient condition for the existence of a linear jump in the order of the largest component in $\mathbf{G}(\cD_n)_p$.

\begin{theorem}[Linear size in wide interval: coarse threshold with linear jumps]\label{thm 2}
Let $\cD_n$ be a degree sequence such that there exist $\delta, \delta'>0$, $d_1 = d_1(n)$ and $d_2 = d_2(n) < d_1$ with $d_1/d_2\ge c$ for some constant $c > 1$, and such that
\begin{enumerate}[i)]
    \item $|S_n(d_2)\setminus S_n(d_1)| \geq \delta n$, and
    \item for some $d \gg d_1$, $|S_n(d)| \geq \delta' n$.
\end{enumerate}
Let $p_1 = \dfrac{1}{d_1}$ and $p_2 = \dfrac{1}{d_2}$. Then, there is a positive real number $C = C(\delta, \delta',c) > 0$ such that 
\begin{equation*}
    \mathbb P\left[L_1(\mathbf{G}(\cD_n)_{p_2}) - L_1(\mathbf{G}(\cD_n)_{p_1}) \ge Cn\right] = 1-o(1).
\end{equation*}
\end{theorem}
A direct consequence of the previous theorem is the following:
\begin{corollary}[Linear size in narrow interval: sharp threshold with linear jumps]\label{thm 3} 
Let $\cD_n$ be a degree\\ 
sequence such that there exist $\delta, \delta'>0$, $ d_1 =  d_1(n)$ and  $ d_2 =  d_2(n)<  d_1$  with $ d_1/d_2\to 1$ such that
\begin{enumerate}[i)]
    \item $|S_n( d_2)\setminus S_n( d_1)|\geq \delta n$, and
    \item for some $d \gg  d_1$, $|S_n(d)| \geq \delta' n$.
\end{enumerate}
Fix $\eps \in (0,1)$ and let $p_1 = \dfrac{1-\eps}{ d_1}$ and $p_2 = \dfrac{1+\eps}{ d_2}\sim \dfrac{1+\eps}{d_1}$. Then, there is a positive real number $C = C(\delta, \delta', \eps) > 0$ such that 
\begin{equation*}
    \mathbb P\left[L_1(\mathbf{G}(\cD_n)_{p_2}) - L_1(\mathbf{G}(\cD_n)_{p_1}) \ge Cn\right] = 1-o(1).
\end{equation*}
\end{corollary}
\noindent
In both Theorem~\ref{thm 2} and Corollary~\ref{thm 3}, the conditions $\delta, \delta'>0$ are necessary, as Remarks~\ref{rem:1} and~\ref{rem:2} in Section~\ref{sec:Remarks} show.

We now focus on sublinear size thresholds.

\begin{theorem}[Sublinear size thresholds]\label{thm 6}
Let $\cD_n$ be a degree sequence and let $d = d(n)$. 
\begin{enumerate}
    \item [(a)] If $|S_n(d)|\to \infty$, for any $p$,
    \begin{align*}
        \bP\bigg[L_1(\mathbf{G}(\cD_n)_p)\le (1+o(1))\bigg(|S_n(d)|+\sum_{v\notin S_n(d)} (1-(1-p)^{d(v)})\bigg)\bigg]=1-o(1).
    \end{align*}
    \item [(b)] if $d(S_n(d))\gg d(V\setminus S_n(d))$, for any $p\gg 1/d$, a.a.s.\ there is a connected component in $\mathbf{G}(\cD_n)_p$ containing a $(1-o(1))$-proportion of the edges of $\mathbf{G}(\cD_n)_p$.
\end{enumerate} 
\end{theorem}
Observe that Part~(b) of Theorem~\ref{thm 6} has a somehow different flavor from the other results since it provides a component containing many edges. 
Although in general we cannot say directly how many vertices this dense component has, in case a $(1-o(1))$-proportion of all vertices have degrees of the same order, it is equivalent to the existence of a component containing almost all vertices.

We complement the previous theorem with the following result that holds for any deterministic graph.
\begin{theorem}\label{thm 7}
Let $\cD_n$ be a degree sequence and let $d = d(n)$. Fix any $p\ll 1/d$ such that $p d(S_n(d)) \gg 1$ and let $\omega = \min\{(dp)^{-1}, p d(S_n(d))\}$. For any graph $G$ with vertex set $[n]$ and degree sequence $\cD_n$,
    \begin{align*}
        \bP[L_1(G_p)\le \max\{|S_n(d)|+(1+2\omega^{-1/3})p d(S_n(d)), 2\tfrac{\log n}{\log\omega}\}]=1-o(1).
    \end{align*}
\end{theorem}

Our final result gives a rather surprising example that there exist degree sequences for which the order of the largest component after $p$-percolation undergoes unboundedly many jumps (as a function of $p$). 
\begin{theorem}\label{thm example}
Fix any $k \in \mathbb N$. There exist an increasing sequence of non-negative real numbers $(\alpha_i)_{0 \le i \leq k}$ and two decreasing sequences of positive real numbers $(\beta_i)_{1 \le i \le k}$ and $(\delta_i)_{1 \le i \le k}$ such that the following holds:
Let $\cD_n^k$ be the degree sequence consisting of $n^{\alpha_i} - n^{\alpha_{i-1}}$ vertices of degree $n^{\beta_i}$ for every $i\in [k]$, and $n - n^{\alpha_k} +1$ vertices of degree $1$, and let $\mathbf{G} = \mathbf{G}(\cD_n^k)$. For every $i\in [k]$, a.a.s.
\begin{itemize}
    \item[(a)] if $p\ll n^{1-\alpha_i-2\beta_i}$, then $L_1(\mathbf{G}_{p}) = O( n^{1-\beta_{i} - \delta_i})$.
    \item[(b)] if $p\gg n^{1-\alpha_i-2\beta_i}$, then $L_1(\mathbf{G}_{p})\ge  n^{1-\beta_i}$.
\end{itemize}
\end{theorem}

The previous theorem shows that there exist sequences having unboundedly many jumps in the order of the largest component. 
In fact, we also provide an example of sequences with unboundedly many jumps in the leading constant, see Section~\ref{sec:exm1}. A similar phenomenon has been observed by Riordan and Warnke in the context of the Achlioptas processes~\cite{riordanWarnke}.

\paragraph{Main ideas and plan of the paper.}
Since the focus of this paper is on dense degree sequences, classical methods based on the configuration model may not be applied in general, as a dense random graph might have a superexponentially small probability (in terms of the number of vertices) to be simple. 
Our main tool is the switching method; however, analysis of multi-type branching processes and several large deviation inequalities from probability theory, such as Janson's and Talagrand's inequalities, are also important in our arguments.

The paper is organized as follows. In Section~\ref{sec:tech} we introduce the switching method and other probabilistic tools. In Section~\ref{sec:mainUB} we provide the proofs of the upper bounds in Part~(a) of Theorems~\ref{thm 1} and~\ref{thm 6}, and in Theorem~\ref{thm 7}. 
The proofs of the lower bounds in Part (b) of Theorems~\ref{thm 1} and~\ref{thm 6}, and results on component jumps in Theorems~\ref{thm 4},~\ref{thm 2} and Corollary~\ref{thm 3} are then presented in Section~\ref{sec:mainLB}.
Since we reuse parts of the proofs of different theorems, the order in which the theorems are proven is different from the order in which they are stated above. Note that some proofs show slightly strengthened versions of the results in this section. In Section~\ref{sec:exm2} we present two examples of multi-jump behavior and, in particular, prove Theorem~\ref{thm example}.

\section{Preliminaries and tools}\label{sec:tech}
In this section we introduce the main technical tools used in our proofs.

\subsection{The switching method}\label{sec:switch}

Our key technique to prove a.a.s. statements is a simple combinatorial switching argument. In this section we describe the switchings that we will use. 

For any graph $G$ and four vertices $a, b, c, d\in V(G)$, provided that $ab, cd\in E(G)$ and $ad, bc\notin E(G)$ one may \emph{switch} the edges $ab$ and $cd$ along the 4--cycle $abcd$ to obtain the graph $G'$ with the same vertex set and with edge set $E(G')=(E(G)\cup\{ad,bc\})\setminus \{ab,cd\}$. Note that switching operations are reversible: if $G'$ is obtained from $G$ by switching the edges $ab$ and $cd$ to get $ad$ and $bc$, then $G$ is obtained from $G'$ by switching back the edges $ad$ and $bc$.

Switchings were introduced by Petersen in~\cite{Pet91}. The use of switchings to analyze graphs with prescribed degrees was pioneered by McKay~\cite{McK81}. For more information about the historical background see~\cite{Wor99}.

Let $\cG(\cD_n)$ be the set of simple graphs on $V$ with degree sequence $\cD_n$. Recall that any switching operation preserves the degree of each vertex, so any switching of a graph in $\cG(\cD_n)$ yields a graph in $\cG(\cD_n)$. The switching method allows us to compare the sizes of two subsets $\cE,\cF\subseteq \cG(\cD_n)$. Build an auxiliary bipartite multigraph with partite sets $\cE$ and $\cF$, where we add an edge between $G\in \cE$ and $G'\in \cF$ for every switching that transforms $G$ into $G'$, or equivalently $G'$ into $G$. Given a lower bound $d^-_\cE$ on the degrees of any $G \in \cE$ and an upper bound $d^+_\cF$
on the degrees of any $G' \in \cF$, a simple double-counting argument implies that
$$
d^-_\cE |\cE| \le d^+_\cF|\cF|.
$$
Since $\mathbf{G}(\cD_n)$ is uniform in $\cG(D_n)$, by abuse we may consider $\cE$ and $\cF$ as events in the probability space $\mathbf{G}(\cD_n)$. The previous inequality implies that
$$
\Pr(\cE) \le \frac{d^+_\cF}{d^-_\cE}\cdot \Pr(\cF).
$$
Our main object of interest is $\mathbf{G}(\cD_n)_p$. It is useful to imagine $\mathbf{G}(\cD_n)_p$ as the graph $\mathbf{G}(\cD_n)$ equipped with a 2--coloring of its edges: edges in blue correspond to the ones that percolate and edges in red correspond to the ones that do not percolate. The probability that we obtain any given edge-coloring of a graph $G\in \cG(\cD_n)$ with $\ell$ blue edges is
 \begin{align}\label{eq:unif_col}
 \frac{p^{\ell}(1-p)^{m/2-\ell}}{|\cG(\cD_n)|} = \frac{(1-p)^{m/2}}{|\cG(\cD_n)|}\left(\frac{p}{1-p}\right)^{\ell},
 \end{align}
with $m$ being equal to the sum of all degrees.
So, conditionally on the number of blue edges, we get a uniform random graph with degree sequence $\cD_n$ equipped with a uniform random 2--coloring with $\ell$ blue edges.
 
Switchings can be extended to edge-colored graphs. To be able to analyze $\mathbf{G}(\cD_n)_p$ via swichings, we only use switchings that preserve the number of blue edges. In such case,~\eqref{eq:unif_col} ensures that the probability of an edge-colored graph and the one obtained by applying a switching to it is the same. In this paper it will suffice to consider blue switchings, that is, we restrict the choice of $ab$ and $cd$ to the set of blue edges. To prevent the appearance of multiedges, when switching $ab$ and $cd$ along $abcd$, we will insist that $ad$ and $bc$ are not edges in our graph (that is, neither blue nor red). 


We state a simple lemma on the number of pairs of oriented edges such that after switching them, one obtains additional components.

\begin{lemma}\label{Felix and Guillem}\cite{joos2016how}
The number of pairs of oriented edges $uv, xy$ in a graph $G$ of order $n$ such that, by switching $uv$ and $xy$, we obtain a graph with one more component than $G$, is at most $8n^2$.
\end{lemma}

\subsection{Connection between average and minimum degree in graphs}

In this subsection, we state a version of a classical result that holds for all deterministic graphs, and provide a proof for completeness.

\begin{lemma}\label{lem:avg to min}
Fix an $n$-vertex graph $G$, $d = d(n)$ satisfying $|S_n(d)| = \Omega(n)$ and fix $d_{\rm{min}} < d/2$. Then, $G$ has a subgraph $H$ containing at least 
\[|S_n(d)| - \frac{(n-|S_n(d)|)d_{\min}}{d-2d_{\min}}\]
vertices and minimum degree at least $d_{\rm{min}}$. 
In fact, if $d_{\min} = o(d)$, then one may ensure that $H$ contains $(1-o(1))|S_n(d)|$ vertices and a $(1-o(1))$-proportion of all edges of $G$.
\end{lemma}
\begin{proof}
We perform the following greedy algorithm. 
Let $G_0=G$, $A_0 = \emptyset$ and $B_0 = [n]$. 
At step $i\ge 1$, choose the smallest vertex with degree less than $d_{\rm{min}}$ in $G_{i-1}$. 
If there is no such vertex, set $H=G_{i-1}$, which has minimum degree at least $d_{\rm{min}}$, and terminate. 
Otherwise, let $j_i$ be the smallest vertex in $V(G_{i-1})$ of degree less than $d_{\rm{min}}$. Delete $j_i$ from $G_{i-1}$ to obtain $G_i$, update $A_i = A_{i-1}\cup \{j_i\}, B_i = B_{i-1}\setminus \{j_i\}$ and go to step $i+1$.

Let $f+1$ be the step at which the algorithm terminates, i.e. $H=G_f$. 
At any step $i\ge 0$, the number of edges in the cut $(A_i,B_i)$ increases by less than $d_{\rm{min}}$, so the cut $(A_f,B_f)$ has at most $|A_f| d_{\rm{min}}$ edges. 
On the other hand, if at step $i$ a vertex with degree (in $G$) at least $d$ is chosen as $j_i$, the number of edges in the cut $(A_i,B_i)$ decreases by at least $(d-d_{\rm{min}})-d_{\rm{min}} = d - 2d_{\rm{min}}$. Therefore, the number of steps at which a vertex of degree at least $d$ is deleted is at most
\begin{equation*}
    \dfrac{(n-|S_n(d)|)d_{\min}}{d - 2d_{\min}},
\end{equation*}
which proves the first part. 
The second part follows from the fact that, if $d_{\min} = o(d)$, then only $o(n) = o(|S_n(d)|)$ vertices and a negligible proportion of the edges can be deleted in the above procedure.
\end{proof}

\begin{remark}\label{rem:unifH}
We remark that, if the graph $G$ in Lemma~\ref{lem:avg to min} is a random graph that is uniformly distributed on its degree sequence, so is $H$. 
We prove this by induction (with the notation from the proof of Lemma~\ref{lem:avg to min})
on $i\in [0, f]$. 
For $i = 0$, the observation is trivial. 
Suppose that the induction hypothesis holds at step $i-1$. 
Since $G_{i-1}$ has a uniform distribution by the induction hypothesis, by conditioning on the edges (and non-edges) incident to $j_i$, 
the remaining graph must be uniformly distributed over the set of all graphs on $B_{i-1}$ and the degree sequence $\cD_{i-1}$ of $G_{i-1}$ 
respecting the conditioning since each of them corresponds to exactly one graph including $j_i$, 
and each such graph is chosen with the same probability. 
Thus, deleting the vertex $j_i$ leads to a uniform random graph $G_i$ on vertex set $B_i$ with the corresponding degree sequence, as desired.
\end{remark}

\subsection{Stochastic comparison of the components of vertices with different degrees}

In this section we show a general preliminary lemma saying that a vertex with higher degree will have a higher probability to be in a larger component of $\mathbf{G}(\cD_n)_p$.

\begin{lemma}\label{lem dom}
Fix $u,v\in V$ with $d(u)\le d(v)$ and $p\in[0,1]$. Let $\cC_u$ and $\cC_v$ be the components in $\mathbf{G}(\cD_n)_p$ containing $u$ and $v$, respectively. Then, $|V(\cC_u)|$ is stochastically dominated by $|V(\cC_v)|$.
\end{lemma}
\begin{proof}
We may assume that $uv\notin \mathbf{G}(\cD_n)$; otherwise, the argument is analogous by assuming that $uv\notin E(\mathbf{G}(\cD_n)_p)$ and reducing both degrees by 1 (clearly, if $uv$ is an edge that percolates, $\cC_u = \cC_v$). 
Choose $S\subseteq V\setminus\{u,v\}$ with $|S|\in [d(v),d(u)+d(v)]$, and $T\subseteq S$ with $|T|=d(u)+d(v)-|S|$. 
Let $A_{S,T}$ be the event in $\mathbf{G}(\cD_n)$ that $S=N(u)\cup N(v)$ and $T=N(u)\cap N(v)$. 
Conditionally on $A_{S,T}$, expose the graph $\mathbf{G}(\cD_n)_p$ induced by $V\setminus \{u,v\}$.  
Given $\mathbf{G}(\cD_n)_p\setminus \{u,v\}$, the order of the components of $u$ and $v$ in $\mathbf{G}(\cD_n)_p$ is determined by the state of the edges between $S$ and $u$ and between $S$ and $v$, respectively. 
A simple counting argument implies that, for every $N\subseteq S$ and a degree sequence with $d(u) \le d(v)$, the probability that $N\subseteq N(u)$ in $\mathbf{G}(\cD_n)$ is at most the probability that $N\subseteq N(v)$ in $\mathbf{G}(\cD_n)$. 
It follows that we can couple $\cC_u$ and $\cC_v$ such that $|V(\cC_u)|\leq |V(\cC_v)|$, as desired.
\end{proof}

\subsection{Further probabilistic ingredients}

In this subsection we gather several large deviation probabilistic inequalities.

\begin{lemma}[Chernoff's inequality, see e.g.~\cite{JLR}]\label{chernoff}
Let $(X_i)_{i=1}^n$ be independent random variables taking values in $\{0,1\}$, and define $X = X_1 + X_2 + \dots + X_n$ and $\mu = \mathbb E[X]$. Then,
\begin{align*}
    & \mathbb P(X_1 + \dots + X_n \le (1-\delta)\mu) \le \exp\left(-\dfrac{\delta^2 \mu}{2}\right)\text{ for every } \delta\in [0,1], \; and\\
    & \mathbb P(X_1 + \dots + X_n \ge (1+\delta)\mu) \le \exp\left(-\dfrac{\delta^2 \mu}{2+\delta}\right)\text{ for every } \delta\ge 0.
\end{align*}
\end{lemma}

\noindent 
Next, we say that a function $f: \mathbb{R}^n \to \mathbb{R}$ is \emph{$c$-Lipschitz} if for every $i\in [n]$ and any real numbers $(z_j)_{j=1}^n$ and $z'_i$ we have
$$
|f(z_1,\dots,z_i,\dots, z_n)- f(z_1,\dots,\hat{z}_i,\dots, z_n)| \leq c.
$$
We say that a function $f: \mathbb{R}^n \to \mathbb{R}$ is \emph{$r$-certifiable} if whenever $f(z)\geq s$, there exists a set $I\subseteq [n]$ with $|I|\leq rs$ elements such that any $\hat{z} \in \mathbb{R}^n$ agreeing with $z$ on all coordinates in $I$ satisfies $f(\hat z)\geq s$.

\begin{lemma}[Talagrand's inequality for certifiable Lipschitz functions, see e.g.~\cite{molloy2002graph}]\label{thm:McD2}
Let $(Z_i)_{i=1}^n$ be independent random variables and $f:\mathbb R^n\to \mathbb R$ be a $c$-Lipschitz and $r$-certifiable function. Denote $X=f(Z_1,\dots,Z_n)$. Then, for every $t>0$,
$$
\Pr(|X- \bE[X]| > t + 60c\sqrt{r\bE[X]})\leq 4\exp\left(-\frac{t^2}{8c^2r\bE[X]}\right).
$$
\end{lemma}

We will also use two inequalities for product measures on powersets. Let $E$ be a finite set and $2^{E}$ its powerset. For $q\in [0,1]$, let $\Pr_q$ be the $q$-product measure on $2^{E}$, that is, $\Pr_q(A)=q^{|A|}(1-q)^{|E\setminus A|}$ for all $A\subseteq E$. A family $\mathcal A\subseteq 2^E$ is called \emph{increasing} if for every pair of sets $A,B\subseteq E$ satisfying $A\subseteq B$ and $A\in \cA$ we also have $B\in \cA$.

\begin{lemma}[Harris' inequality~\cite{harris1960lower}]\label{harris}
Let $q \in [0,1]$. Any pair of increasing families $\cA,\cB \subseteq 2^E$ satisfies
$$
\Pr_q(\cA)\Pr_q(\cB)\leq \Pr_q(\cA\cap \cB).
$$
\end{lemma}

\begin{lemma}[A version of Janson's inequality, see, e.g.,~\cite{JLR}]\label{janson}
Let $E_q$ be a $\Pr_q$-random subset of $E$ and let $\{F_i\}_{i\in I}$ be a collection of subsets of $E$. For $i\in I$, let $X_i=\mathds 1_{F_i\subseteq E_q}$, $X=\sum_{i\in I} X_i$ and $\mu = \mathbb E[X]$.
Define also $\mathcal T = \{(i,j)\in \mathcal I^2\mid F_i\cap F_j\neq \emptyset\}$ and let $\Delta = \sum_{(i,j)\in \mathcal T} \bE[X_iX_j]$. Then, for every $t\in [0,\mu]$,  
\begin{equation*}
    \Pr_q(X\le \mu-t)\le \exp\left(-\frac{t^2}{2(\mu+\Delta)}\right).
\end{equation*}
\end{lemma}

Finally, for an integrable probability distribution $\mu$ on the set of non-negative integers with expectation $\mathbb E[\mu]$, we recall that the \emph{Bienaim\'e-Galton-Watson tree} with offspring distribution $\mu$ is constructed as follows: starting from a single vertex (the root), every vertex produces a number of children distributed according to $\mu$ and independent from the offspring of all other vertices. In particular, defining $X_k$ as the number of vertices at graph distance $k\ge 1$ from the root, we get that $\mathbb E[X_k] = \mathbb E[\mathbb E[X_k\mid X_{k-1}]] = \mathbb E[\mu]\cdot \mathbb E[X_{k-1}] = \ldots = \mathbb E[\mu]^k$.

\section{Proofs of the main results: upper bounds}\label{sec:mainUB}

In this section we provide the proofs of Part~(a) of Theorems~\ref{thm 1} and~\ref{thm 6} in Section~\ref{sec:thms16}, and the proof of Theorem~\ref{thm 7} in Section~\ref{sec:thm7}.

\subsection{Proof of Part~(a) in Theorems~\ref{thm 1} and~\ref{thm 6}}\label{sec:thms16}

Fix $d\in [n-1]$ and $p\in [0,1]$. Write $S=S_n(d)$ and let $X$ be the number of vertices in $V\setminus S$ that do not remain isolated in $\mathbf{G}(\cD_n)_p$. Then,
\begin{align*}
    \bE[X]= \sum_{v\in V\setminus S} (1-(1-p)^{d(v)}) = N_1(d,p).
\end{align*}
Observe that Part~(a) of Theorem~\ref{thm 1} is implied by Part~(a) of Theorem~\ref{thm 6} since $N_1(d,p)=o(n)$ for $p\ll 1/d$. Therefore, it suffices to prove the latter. 
In fact, we will show that a.a.s.\ the number of non-isolated vertices is dominated by the upper bound in Part~(a) of Theorem~\ref{thm 6}. To do so, it suffices to show that typically $X - \mathbb E[X] = o(\mathbb E[X])$ conditionally on any realization of $G\in \cG(\cD_n)$.
 
Given $G$, $X$ is a function $f$ of the independent random variables that mark whether an edge has percolated or not. The function $f$ is $2$-Lipschitz (changing the state of an edge may only change the fact that its endvertices are isolated or not) and $1$-certifiable (the fact that there are $s\ge 1$ non-isolated vertices can be witnessed by $s$ edges). 
Let $\omega=\omega(n)$ be a function that tends to infinity sufficiently slowly. By Talagrand's inequality (Lemma~\ref{thm:McD2}) with $t=\sqrt{\omega \bE[X]}\ge 120\sqrt{\bE[X]}$, we obtain
\begin{align}\label{eq:XLB}
\Pr(X > N_1(d,p)+2\sqrt{\omega N_1(d,p)})\leq \Pr(|X- \bE[X]| >2\sqrt{\omega\bE[X]}])\leq 4\exp\left(-\frac{\omega}{32}\right)=o(1).
\end{align}
Since this bound holds for any choice of $G\in \cG(\cD_n)$, it also holds unconditionally. 

Note that if $N_1(d,p)\to \infty$ as $n\to \infty$, then~\eqref{eq:XLB} implies that $X\le (1+o(1)) N_1(d,p)$ a.a.s., whilst if $N_1(d,p) = O(1) = o(|S|)$, Markov's inequality implies that $X = o(|S|)$ as well. 
Since there are at most $|S|$ vertices of degree at least $d$, a.a.s.\ the total number of vertices in the largest component of $\mathbf{G}(\cD_n)_p$ is at most 
\begin{align*}
    |S|+X\leq (1+o(1))(|S|+N_1(d,p)),
\end{align*}
thus proving Theorem~\ref{thm 6}~(a).

\subsection{Proof of Theorem~\ref{thm 7}}\label{sec:thm7}

Let $S=S_n(d)$ and recall the assumptions $pd\ll 1\ll pd(S)$. Denote by $\cC$ the union of the components of $G_p$ that intersect $S$. We will show that a.a.s.\
\begin{enumerate}
    \item $|V(\cC)|\leq |S|+(1+2\omega^{-1/3})p d(S)$, and
    \item $L_1(G_p)\leq \max\left\{|V(\cC)|, 2\tfrac{\log n}{\log \omega}\right\}$.
\end{enumerate}
Let $H=G[V\setminus S]$ be the subgraph induced by the vertices of degree less than $d$. For any $v\notin S$, let $\cH(v)$ be the component of $v$ in $H_p$. 
\begin{lemma}\label{lem:BP}
For any $v\notin S$, we have $\bE[|\cH(v)|]\leq 1+2\omega^{-1}$. Moreover, for all $\mu\ge 0$,
\begin{align}\label{eq:proba}
    \bP\Big(|\cH(v)|\geq \mu\frac{\log n}{\log \omega}\Big) =o(n^{-(1+o(1))\mu}).
\end{align}
In particular, a.a.s.\ $L_1(H_p)\leq 2\frac{\log n}{\log \omega}$.
\end{lemma}
\begin{proof}
We consider a breadth first search (BFS) exploration process of $\cH(v)$ starting at $v$.\footnote{The formal description of the exploration process goes as follows: first, assign label 0 to the vertex $v$. Then, as a first step, reveal the neighbors of $v$ and assign them labels between 1 and $d(v)$ in an arbitrary order. Suppose that just before step $i$ one has revealed vertices with labels between 0 and $\ell_i$. If $\ell_i < i$, then finish the exploration process. Otherwise, reveal all neighbors of $i$, and if exactly $r_i$ of them have not been labeled before, assign labels $\ell_i+1,\dots,\ell_i+r_i$ to these new vertices in an arbitrary order.} Since $d(v)\leq d$, we can couple the above BFS process with a Bienaym\'e-Galton-Watson tree with offspring distribution $\mathrm{Bin}(d,p)$ in such a way that the latter stochastically dominates the former. For every $k\ge 0$, we denote by $X_k$ the size of the $k$-th generation of the Bienaym\'e-Galton-Watson tree and let $T=\sum_{k\geq 0} X_k$ be the total progeny of the tree. Then, $|\cH(v)|\le |T|$ holds on the probability space where $\cH(v)$ and $T$ have been coupled.

By the choice of $\omega$, we have $\bE[X_k]\le (dp)^k\leq \omega^{-k}$ for all $k\in\mathbb{N}$. It follows that 
\begin{align*}
    \bE[|\cH(v)|]\leq \sum_{k\geq 0} \bE[X_k]\leq 1+2\omega^{-1}.
\end{align*}
In order to bound the probability that $\cH(v)$ is large, we will use classical results for branching processes with binomial offspring. 
Let $I_\lambda:=\lambda-1-\log \lambda$ denote the large deviation function for Poisson random variables with mean $\lambda$ and note that $I_{1/\omega}= (1+o(1))\log \omega$. 
Since $dp\leq 1/\omega$, by Corollary 2.20 in~\cite{vdH17} we deduce that for any $t\ge 0$,
\begin{align*}
    \bP(|\cH(v)|\ge t)\leq \bP(T\ge t)\leq e^{-t I_{1/\omega}}.
\end{align*}
Choosing $t = \mu\frac{\log n}{\log \omega}$ finishes the proof of the first statement. Moreover, by choosing $\mu=2$ and taking a union bound over all at most $n$ vertices, we deduce that a.a.s.\ $L_1(H_p)\le 2\frac{\log n}{\log \omega}$, as desired.
\end{proof}

Let $N_G(S)$ be the set of vertices $v\notin S$ with at least one neighbor in $S$. For $v\in N_G(S)$, let $\mathcal E_p(v, S)$ denote the event that there exists at least one edge between $v$ and $S$ in $G_p$. The number of vertices $v$ for which $\mathcal E_p(v, S)$ holds, is stochastically dominated by $\text{Bin}(d(S),p)$. Thus, $|V(\cC)|$ is stochastically dominated by the sum of $|S|$ and two random variables: $X \sim \text{Bin}(d(S),p)$ and
$$
Y=\sum_{\substack{e\in E(S,V\setminus {S})\\ e=uv}}\mathds{1}_{e\in E(G_p)}(|\cH(v)|-1).
$$
Indeed, $X$ dominates stochastically the number of vertices in $N_G(S)$, and $Y$ counts for the orders of the components of the vertices $v\in N_G(S)$ in $H_p$ without counting $v$ itself (since $v$ is already ``counted'' by $X$). By Lemma~\ref{lem:BP}, $\bE[Y]\leq 2\omega^{-1} \bE[X]\le 2\omega^{-1} p d(S)$, and hence Markov's inequality yields a.a.s.\ $Y\leq \omega^{-1/3} p d(S)$. 
Since $p d(S)\to \infty$, by Chernoff's inequality a.a.s.\ $|X-\mathbb E[X]| = o(\mathbb E[X]^{2/3}) = o(\omega^{-1/3} pd(S))$, and hence a.a.s.
\begin{align*}
    |V(\cC)|\leq |S|+\mathbb E[X]+2\omega^{-1/3} p d(S) = |S| + (1+2\omega^{-1/3}) p d(S).
\end{align*}

By Lemma~\ref{lem:BP} the order of the components in $G_p$ that do not intersect $S$ is a.a.s.\ at most $2\frac{\log n}{\log \omega}$, which concludes the proof of Theorem~\ref{thm 7}.

\section{Proof of the main results: lower bounds}\label{sec:mainLB}

We continue by providing a proof of Part~(b) in Theorems~\ref{thm 1},~\ref{thm 4} and~\ref{thm 6} in Sections~\ref{sec:thms145_b} and~\ref{thm:19b}. We then prove Theorem~\ref{thm 2} and Corollary~\ref{thm 3} in Section~\ref{sec:thm23}.

\subsection{Proof of Part~(b) in Theorems~\ref{thm 1} and~\ref{thm 4}}\label{sec:thms145_b}

First, we show Theorem~\ref{thm 4} assuming Theorem~\ref{thm 1}.

\begin{proof}[Proof of Theorem~\ref{thm 4} assuming Theorem~\ref{thm 1}]
If $|S_n(\omega^{-1} d_2)| = o(n)$, then by Theorem~\ref{thm 1}~(a) with $d = d_2/\omega$ and $p = p_2\ll 1/d$ we get that $L_1(\mathbf{G}(\cD_n)_{p_2}) = o(n)$ a.a.s. If not, suppose for contradiction that there is $\eps > 0$ and an increasing subsequence $(n_i)_{i\ge 1}$ satisfying
\begin{equation}\label{eq:ass_th25}
\mathbb P(L_1(\mathbf{G}(\cD_{n_i})_{p_2}) - L_1(\mathbf{G}(\cD_{n_i})_{p_1})\ge \eps n_i)\ge \eps.
\end{equation}
By assumption, there is $\delta > 0$ for which one may further extract an increasing subsequence $(n'_i)_{i\ge 1}$ of $(n_i)_{i\ge 1}$ such that $|S_{n'_i}(\omega^{-1} d_2)| = \delta n'_i + o(n_i)$. Hence, by Theorem~\ref{thm 1}~(a) with $d = \omega^{-1} d_2$ and $p = p_2\ll 1/d$ we get that $\mathbb P(L_1(\mathbf{G}(\cD_{n'_i})_{p_2})\le \delta n'_i + o(n'_i)) = 1-o(1)$, and by Theorem~\ref{thm 1}~(b) with $d = \omega d_1$ and $p = p_1\gg 1/d$ we get that $\mathbb P(L_1(\mathbf{G}(\cD_{n'_i})_{p_1})\ge \delta n'_i - o(n'_i)) = 1-o(1)$. Combining the above two applications of Theorem~\ref{thm 1} leads to a contradiction of our assumption of~\eqref{eq:ass_th25}.
\end{proof}

Now, let $dp = \omega^3$ for some function $\omega = \omega(n)\to \infty$. 
By Lemma~\ref{lem:avg to min} we know that $\mathbf{G}(\cD_n)$ contains a graph $\mathbf{H}$ with minimum degree at least $d' := d/\omega$ which contains almost all vertices of $|S_n(d)|$ and is also uniformly distributed on its degree sequence by Remark~\ref{rem:unifH}.

To the end of the section, we show that a.a.s. the graph $\mathbf{H}_p$ contains a component of order $(1-o(1))h$ where $h := |V(\mathbf{H})|$. 
To this end, we show that $\mathbf{H}$ has good expansion properties and all cuts $(V_1, V_2)$ in $\mathbf{H}_p$ with $|V_1|, |V_2|\ge h/\omega$ contain at least one edge, which clearly implies Part~(b) in Theorem~\ref{thm 1}. 
To do so, we show that a.a.s.\ every such cut in $H$ contains many edges and at least one of these edges is retained in $\mathbf{H}_p$.

\begin{lemma}\label{lem:cutH}
For every $s\in [h/\omega, h/2]$, with probability at least $\exp(-\Omega(d's))$, every cut $(V_1, V_2)$ in $\mathbf{H}$ with $|V_1| = s$ contains at least $d's(h-s)/(200h)$ edges. 
\end{lemma}
\begin{proof}
The proof relies on the first moment method. 
Fix a cut $(V_1, V_2)$ of $\mathbf{H}$ with $|V_1|=s$. Let $k_0:=d's(h-s)/(100h)-2$. For any $k\leq k_0$, denote by $\mathcal F_k$ the event $e(V_1,V_2) = k$. 
The number of switchings from $\mathcal F_{k+2}$ to $\mathcal F_k$ is clearly at most $2\cdot (k+2)(k+1)/2$.

Let us bound from below the number of switchings from $\mathcal F_k$ to $\mathcal F_{k+2}$. 
On the one hand, the subset of $V_2$ consisting of the vertices with at least $s/3$ neighbors in $V_1$ has size at most $3k/s < d'/30$. 
Since each such vertex has degree at most $h-s$ inside $V_2$, the graph induced by its complement (which we denote by $W_2\subseteq V_2$) contains at least
$d'(h-s)/2 - d'(h-s)/30 -k > d'(h-s)/3$ edges. 
At the same time, since the cut has $k$ edges, there are at most $(h-s)/25$ vertices $v\in V_2$ satisfying $e(\{v\}, V_1)\ge 25k/(h-s)$, or alternatively, at least $24(h-s)/25$ vertices in $V_2$ have at most $25k/(h-s) < d's/(4h)$ neighbors in $V_1$. Call this set $U_2$ and note that $U_2\subseteq W_2$ since $d's/(4h) < s/3$. 

Fix one edge $uw$ with $u\in U_2$, $w\in W_2$.
The number of switchings including $uw$ and increasing the cut is at least the number of edges with one endvertex in $V_1\setminus (N(u)\cup N(w))$ and one endvertex in  
$V_1\setminus (N(u)\cap N(w))$.
On the one hand, $|N(u)\cap N(w)\cap V_1|\le |N(u)\cap V_1|\le d's/(4h)$, so there are at most $d's^2/(4h) \le d's/8$ 
edges in $\mathbf{H}[V_1]$ incident to $N(u)\cap N(w)\cap V_1$.
Hence, the number of edges in $\mathbf{H}[V_1]$ incident to a vertex in $V_1\setminus (N(u)\cup N(w))$ but not to $N(u)\cap N(w)\cap V_1$ is at least 
\[\frac{d'(|V_1|-|N(u)\cup N(w)|)}{2} - k - \frac{d's}{8} \ge \frac{d'(s - 2s/3)}{2} - \frac{d's}{100} - \frac{d's}{8} \ge \frac{d's}{50}.\]
Since the number of edges between $U_2$ and $W_2$ is at least 
$d'|U_2|/2 - (h-s)d'/30 - k > 10 d'(h-s)/25$, we get that 
\[\frac{|\cF_{k}|}{|\cF_{k+2}|}\le \frac{(k+2)^2}{(d's/50)\cdot (10 d'(h-s)/25)}\le \frac{s(h-s)}{80h^2}\le \frac{1}{80}.\]
As a result, $|\cF_{(k_0+2)/2}|/|\cF_{k_0}|\le 80^{-k_0/4}$. Since the number of cuts is at most $2^h$, using that $h-s\ge s \ge h/\omega$ and $d'\ge \omega^2$ (so $k_0 = \Omega(d's)\gg h$) finishes the proof of the lemma.
\end{proof}

We are ready to finish the proof of Theorem~\ref{thm 1}.

\begin{proof}[Proof of Part~(b) in Theorem~\ref{thm 1}]
We condition on the event from Lemma~\ref{lem:cutH}. 
Then, by a union bound, there is a cut $(V_1, V_2)$ in $\mathbf{H}$ with $|V_1| = s\in [h/\omega, h/2]$ which is empty in $\mathbf{H}_p$ with probability at most
\[\sum_{s=h/\omega}^{h/2} \binom{h}{s} (1-p)^{d's(h-s)/(200h)} = h\cdot 2^h\cdot \exp(-\Omega(pd'h/\omega)) = \exp(-\Omega(pd'h/\omega)),\]
where we used that $\tbinom{h}{s}\le 2^h$ and $h = o(pd's)$ for every $s\in [h/\omega, h/2]$.
Combining this with the error bound from Lemma~\ref{lem:cutH} we get that, with probability $1-\exp(-\Omega(pd'h/\omega)) = 1-\exp(-\Omega(\omega n))$, the graph $\mathbf{H}_p$ contains a connected component of order $(1-1/\omega)h$ and almost all vertices of $S_n(d)$ are present in this component.
\end{proof}

\subsection{Proof of Part~(b) in Theorem~\ref{thm 6}}\label{thm:19b}

We are ready to provide the proof of Part~(b) in Theorem~\ref{thm 6} without further preparation.

\begin{proof}[Proof of Part~(b) in Theorem~\ref{thm 6}]
Fix $pd = \omega^3$ for some function $\omega(n)\to \infty$, and set $d' = d/\omega$. 
Again, we construct a subgraph $\mathbf{H}$ with $h = |V(\mathbf{H})|$ from $\mathbf{G}(\cD_n)$ by iteratively deleting the vertices with degree less than $d'$.

Otherwise said, $\mathbf{H}$ contains $(1-o(1))$-proportion of all edges in $\mathbf{G}(\cD_n)$.
Moreover, since $d'p\gg 1$ and $\mathbf{H}$ has minimum degree $d'$, we conclude as in the proof of Part~(b) in Theorem~\ref{thm 1} that a.a.s.\ all but at most $O(\omega^{-1}h)$ vertices are in the same connected component of $\mathbf{H}_p$.
Denote the set of vertices that do not participate in this component by $U$.

Now, by Lemma~\ref{lem dom} (applied with $\mathbf{H}$ instead of $\mathbf{G}(\cD_n)$ and $p$ instead of $q$) implies that the expected average degree (in $\mathbf{H}_p$) of the vertices in $U$ is smaller than the expected average degree of $\mathbf{H}_p$.
Hence, the expected number of edges in $\mathbf{H}_p$ that are incident to a vertex in $U$ is at most $\mathbb E[2|E(\mathbf{H}_p)|\cdot |U|/h] = o(mp)$, and an application of Markov's inequality finishes the proof.
\end{proof}

\subsection{Proof of Theorem~\ref{thm 2} and Corollary~\ref{thm 3}}\label{sec:thm23}

Recall the statements of Theorem~\ref{thm 2} and Corollary~\ref{thm 3}. We note that $d_1 = o(n)$ since $S_n(d)\neq \emptyset$~and $d\gg d_1$. Moreover, we may (and do) assume that $d_1=O(d_2)$: for Corollary~\ref{thm 3}, this assumption obviously holds. For Theorem~\ref{thm 2}, otherwise, let $d_3$ ($d_4$, respectively) be the smallest integer satisfying $|S(d_2)\setminus S(d_3)|\geq \delta n/3$ ($|S(d_2)\setminus S(d_4)|\geq 2\delta n/3$, respectively). In particular, $d_2\le d_3\le d_4\le d_1$. Then, if either $d_3/d_2$, $d_4/d_3$, or $d_1/d_4$, is of order $O(1)$, we may prove the theorem with $d_2$ and $d_3$, with $d_3-1$ and $d_4$, or with $d_4-1$ and $d_1$ respectively, instead of $d_1$ and $d_2$. Otherwise, $d_2\ll d_3\ll d_4\ll d_1$, and one may deduce Theorem~\ref{thm 2} with $C = \delta/4$ from a combination of Theorem~\ref{thm 1}~(a), applied with $d = d_4$ and $p = 1/d_1\ll 1/d_4$, and of Theorem~\ref{thm 1}~(b), applied with $d = d_3$ and $p = 1/d_2\gg 1/d_3$.

To prove Theorem~\ref{thm 2}, we will perform a 2-stage exposure of the edges of $\mathbf{G}_2=\mathbf{G}(\cD_n)_{p_2}$. 
First, we percolate the edges with probability $p_1$, thereby obtaining $\mathbf{G}_1=\mathbf{G}(\cD_n)_{p_1}$. 
It is useful to think about $\mathbf{G}_1$ as a randomly colored copy of $\mathbf{G}(\cD_n)$, where we assign blue to the edges that percolated and red to the other ones. 
We then give another chance to each red edge of $\mathbf{G}_1$ to appear with probability $p=\frac{p_2-p_1}{1-p_1}$. 
An edge that failed to percolate in the first stage but did percolate in the second one updates its color to green. 
Then, $\mathbf{G}_2$ is the subgraph containing all blue and green edges of the randomly colored $\mathbf{G}(\cD_n)$.

We will show that $\mathbf{G}_1$ has a largest (blue) component, called $\cC_1$, of linear order and that the set of vertices, incident to ``a lot of'' red edges and no blue edges, is of linear size. Then, we will show that, after the second percolation stage, a constant fraction of these vertices attach to $\cC_1$ via green paths of length at most $2$.

For $i\in \{1,2\}$ write $S_i=S_n(d_i)$. Denote also $S_{iso}=\{v\in S_2\setminus S_1: d_{p_1}(v) = 0\}$, that is, the (random) set of vertices of degrees between $d_2$ and $d_1$ that are isolated in $\mathbf{G}_1$. Theorem~\ref{thm 1}~(b) implies that there exists a component in $\mathbf{G}_1$ of order at least $\delta' n - o(n)$ a.a.s. Therefore, the largest component $\cC_1$ of $\mathbf{G}_1$ has order at least $\delta' n - o(n)$.

For $v\in S_{iso}$, we split the neighborhood of $v$ in $\mathbf{G}(\cD_n)$ into 3 sets:
\begin{itemize}
    \item $N^1(v)$ (type 1): neighbors of $v$ in $V(\cC_1)$;
    \item $N^2(v)$ (type 2): neighbors of $v$ not in $V(\cC_1)$ which have at least $\sqrt{d_2d}$ neighbors in $V(\cC_1)$;
    \item $N^3(v)$ (type 3): neighbors of $v$ not in $V(\cC_1)$ which have less than $\sqrt{d_2d}$ neighbors in $V(\cC_1)$.
\end{itemize}

Let $S^1_{iso}$ be the set of vertices $v\in S_{iso}$ with $|N^1(v)|\geq d_2/3$. Also, let $S^2_{iso}$ be set of vertices $v\in S_{iso}\setminus S_{iso}^1$ with $|N^2(v)|\geq d_2/3$, and let $S^3_{iso}=S_{iso}\setminus (S^1_{iso}\cup S^2_{iso})$, which in particular satisfies $|N^3(v)| \ge d_2/3$ for all $v\in S^3_{iso}$. 

The idea of the proof is to show first that $|S_{iso}|$ is linear in $n$. 
Next, we show that a.a.s.\ a constant proportion of the vertices in $S^1_{iso}\cup S^2_{iso}$ attach to $\cC_1$ using green edges. 
Finally, we prove that $S^3_{iso}$ is small. 
Therefore, the order of the largest component increases by a constant fraction between the first and the second stage, proving both theorems.

For this purpose, define the events
\begin{align*}
    \cA_0 &= \{|S_{iso}|\ge \delta n/8\}, \quad
    \cA_1 = \{|S^1_{iso}|\ge |S_{iso}|/3\},\quad \text{and} \quad
    \cA_2 = \{|S^2_{iso}|\ge |S_{iso}|/3\}.
\end{align*}

Note that these events are measurable with respect to the $\sigma$-algebra generated by $\mathbf{G}_1$ (i.e., the first stage of the percolation).

\begin{lemma}\label{sublem 2.16}
We have $\Pr(\cA_0)=1-o(1)$.
\end{lemma}
\begin{proof}
Let $G\in \cG(\cD_n)$. We will bound $\Pr(\cA_0\mid \mathbf{G}(\cD_n)=G)$ uniformly for all $G$ using Janson's inequality (Lemma~\ref{janson}). 
For every $v\in V(G)$, we define the event $A_v=\{d_{p_1}(v) = 0\}$ and the random variables $X_v= \mathds 1_{A_v}$ and $X=\sum_{v\in S_2\setminus S_1} X_v$. 
Note that the event $A_v$ is increasing with respect to the red percolation process.

Independently of the choice of $G$, if $v\in S_2\setminus S_1$, then (using that $d_1\ge d_2+1\ge 2$)
\begin{equation}
    \Pr(A_v)\geq \left(1-\dfrac{1}{d_1}\right)^{d_1}\ge \frac{1}{4},
\end{equation}
and, by hypothesis, $\mu=\mathbb{E}[X]\geq (1-o(1))\frac{\delta n}{4}$.

For any two vertices $u,v\in V(G)$, the random variables $X_u$ and $X_v$ are not independent if and only if $uv\in E(G)$. Thus, since $d_1 = o(n)$, we have 
\begin{align*}
    \Delta 
    &
    =
    \sum_{\substack{u,v\in S_2\setminus S_1;\\ uv\in E(G)}} \bE\left[ X_uX_v\right]
    \le
    d_1n = o(n^2).
\end{align*}

By Lemma~\ref{janson} applied with $t=\mu - \delta n/8 = (1-o(1))\mu/2$, we conclude that
\begin{equation*}
\mathbb P\left(X\le \dfrac{\delta n}{8}\right) 
\le \exp{\left(-\frac{(1-o(1))\mu^2}{8(\mu+\Delta)}\right)} = o(1),
\end{equation*}
which proves the lemma.
\end{proof}

\begin{lemma}\label{trivial ob 2}
There exists $c_1>0$ such that
\begin{equation*}
   \Pr\left(L_1(\mathbf{G}_2) - L_1(\mathbf{G}_1)\ge c_1 \delta n\;\bigg|\; \cA_0\cap \cA_1\right) = 1 - o(1).
\end{equation*}
\end{lemma}
\begin{proof}
Let us condition on a realization of the first stage of the percolation process that satisfies $\cA_0 \cap \cA_1$. We will prove the lemma in this conditional space. The two events imply that $|S^1_{iso}|\geq \delta n/24$. For every $v\in S^1_{iso}$, define the event $B_v$ as the event that at least one red edge in $\mathbf{G}_1$ between $v$ and $V(\cC_1)$ percolates (i.e., becomes green) at the second stage. Since $p = \frac{p_2-p_1}{1-p_1}$ is a decreasing function of $p_1$ and $p_1\le p_2/c$, we have $p\ge \frac{(1-1/c)p_2}{1- p_2/c}\ge \frac{c-1}{cd_2}$.

For $v\in S^1_{iso}$,
\begin{equation*}
   \Pr(B_v)= 1 - \left(1-p\right)^{d_2/3} \geq  1 - \left(1-\dfrac{c-1}{cd_2}\right)^{d_2/3}\geq c_0
\end{equation*}
for some constant $c_0>0$.
Conditionally on the first stage, the collection of events $\{B_v:\, v\in S^1_{iso}\}$ is independent since distinct events depend on disjoint sets of edges.
 
By Chernoff's bound (Lemma~\ref{chernoff}), the probability that at most $c_0\delta n/48$ vertices of $S^1_{iso}$ attach to $V(\cC_1)$ in $G_2$ is at most $e^{-c_0\delta n/192}=o(1)$, so the lemma holds with $c_1=c_0/48$.
\end{proof}

\begin{lemma}\label{trivial ob 3}
There exists $c_2>0$ such that
\begin{equation*}
    \Pr\left(L_1(\mathbf{G}_2) - L_1(\mathbf{G}_1)\ge c_2\delta n\;\bigg|\; \cA_0\cap \cA_2\right) = 1-o(1).
\end{equation*}
\end{lemma}
\begin{proof}
Let us condition on a realization of the first stage of the percolation process that satisfies $\cA_0 \cap \cA_2$. We will prove the lemma in this conditional space. The event $\cA_0 \cap \cA_2$ implies that $|S^2_{iso}|\geq \delta n/24$.

Let $T$ be the set of vertices not in $V(\cC_1)$ with at least one neighbor in $S^2_{iso}$ and at least $\sqrt{d_2 d}$ neighbors in $\cC_1$. If $w\in T\cap S_{iso}$, then $w\in S_{iso}^1$ as $\sqrt{d_2d}\geq d_2/3$. 
Similarly to the proof of Lemma~\ref{trivial ob 2}, we may show that a.a.s. at least $c_0\delta n/48$ vertices in $S^2_{iso}$ are connected (by a green edge) to at least one element of $T$ after the second stage. We condition on this event.

Moreover, the probability that $w\in T$ is not connected (by a green edge) to $\cC_1$ after the second stage is at most
\begin{equation*}
    (1-p)^{\sqrt{d_2d}}\le e^{-p\sqrt{d_2 d}}=o(1).
\end{equation*}
since $p=\Omega(1/d_2)$ and $d_2\ll d$. The collection of these events for $w\in T$ is independent, and independent of the existence of green edges between $S^2_{iso}$ and $T$.

We conclude that the expected number of vertices $v\in S^2_{iso}$ such that none of its neighbors in $T$ are adjacent to $\cC_1$ after the second phase is at most $o(|S^2_{iso}|)=o(n)$. By Markov's inequality it follows that a.a.s. there are at least $c_0\delta n/96$ vertices in $S^2_{iso}$ that connect to $\cC_1$ through a green path of length $2$ in $\mathbf{G}_2$, and the lemma holds with $c_2=c_0/96$.
\end{proof}

The next statement follows directly from Lemmas~\ref{sublem 2.16},~\ref{trivial ob 2} and~\ref{trivial ob 3}.
\begin{corollary}\label{trivial cor}
There exists $c_3>0$ such that
\begin{equation*}
    \Pr\left(\{L_1(\mathbf{G}_2) - L_1(\mathbf{G}_1)\geq  c_3\delta n\} \cap  (\cA_1\cup\cA_2)\right)\geq \Pr(\cA_1\cup \cA_2)-o(1). \qedhere
\end{equation*}
\end{corollary}
\begin{proof}
Let $c_3=\min\{c_1,c_2\}$, where $c_1$ and $c_2$ are the constants appearing in Lemmas~\ref{trivial ob 2} and~\ref{trivial ob 3}, respectively. Write $\cA = \{L_1(\mathbf{G}_2) - L_1(\mathbf{G}_1)\geq  c_3\delta n\}$. Then, by Lemmas~\ref{sublem 2.16},~\ref{trivial ob 2} and~\ref{trivial ob 3},
\begin{align*}
    \Pr(\cA\cap(\cA_1\cup \cA_2))
    &\geq \Pr(\cA\cap \cA_0\cap( \cA_1\cup \cA_2))\\
    &=  \Pr((\cA\cap \cA_0\cap \cA_1)\cup(\cA\cap \cA_0\cap \cA_2)) \\
    &=\Pr(\cA\cap \cA_0\cap \cA_1)+\Pr(\cA\cap \cA_0\cap \cA_2) - \Pr(\cA\cap \cA_0\cap \cA_1\cap \cA_2) \\
    &\ge(1-o(1))\Pr(\cA_0\cap \cA_1)+(1-o(1))\Pr(\cA_0\cap \cA_2) - \Pr(\cA_0\cap \cA_1\cap \cA_2)\\
    &=\Pr(\cA_0\cap \cA_1)+\Pr(\cA_0\cap \cA_2) - \Pr((\cA_0\cap\cA_1)\cap (\cA_0\cap\cA_2)) - o(1)\\
    &=\Pr((\cA_0\cap\cA_1)\cup (\cA_0\cap\cA_2)) - o(1)\\
    &=\Pr(\cA_0\cap(\cA_1\cup \cA_2)) - o(1) = \Pr(\cA_1\cup \cA_2) - o(1),
\end{align*}
which proves the corollary.
\end{proof}

In the remainder of the section, we show that $\Pr(\cA_1\cup \cA_2)=1-o(1)$. Recall that $m$ denotes the sum of the degrees of all vertices in $\mathbf{G}(\cD_n)$.

\begin{lemma}\label{cor dom}
Let $\cC_1$ be the largest (blue) component of $\mathbf{G}_1$ after the first percolation stage. Then, a.a.s. all but $o(m)$ (blue, green or red) edges have two endvertices in $V(\cC_1)$.
\end{lemma}
\begin{proof}
Recall that $d_1\ll d$. 
Let $\lambda=\lambda(n)$ be a function tending to infinity with $n$ such that $\lambda^{13} d_1/d\to 0$. 
Define $d'=d/\lambda$. 
Since $p_1 d'\ge \lambda^{12}$ and $s'=|S_n(d')|\geq |S_n(d)|\geq \delta' n$ (so $d(S_n(d'))\ge \delta' d n\ge d'n$), 
by Part~(b) in Theorem~\ref{thm 1} applied with $p_1$ and $d'$, a.a.s.\ $\mathbf{G}_1$ has a component $\cC_1$ with at least $(1-O(\lambda^{-4}))s'$ vertices. 
Moreover, using that the error bound in Part~(b) in Theorem~\ref{thm 1} is $\exp(-\Omega(n))$ by Remark~\ref{rem:1.4}, the expected number of vertices of $S_n(d')$ outside $\cC_1$ is at most $O(\lambda^{-4} s')$.
As a consequence, for at least $(1-\lambda^{-1})s'$ vertices in $S_n(d')$, 
the probability that they belong to $\cC_1$ has to be at least $1-\lambda^{-1}$: to see this, note that the expected fraction of vertices of $S_n(d')$ belonging to $\cC_1$ would otherwise be at most
\[(1-\lambda^{-1})\cdot 1 + \lambda^{-1}\cdot (1-\lambda^{-1}) = 1 - \lambda^{-2},\]
which is a contradiction with our previous bounds.

Let $d''=d''(n)$ be the smallest integer such that $|S_n(d''+1)| < (1 - \lambda^{-1})s'$. 
By Lemma~\ref{lem dom} one may deduce that, for every $v\in S_n(d'')$, the probability that $v$ belongs to $\cC_1$ is at least $1-\lambda^{-1}$. 
In particular, if an edge has both endpoints of degree at least $d''$, the probability that at least one of its endpoints does not belong to $\cC_1$ is at most $2\lambda^{-1}$. 

Let us show that most of the edges in $\mathbf{G}(\cD_n)$ have both endpoints of degree at least $d''$. First of all, note that $dn=O(m)$ since there are at least $\delta'n+o(n)$ vertices of degree at least $d$. 
On the one hand, the number of edges with at least one endpoint of degree at most $d'$ is at most $d' n = \lambda^{-1} dn=O(\lambda^{-1} m)$. 
On the other hand, the number of vertices of degree between $d'$ and $d''$ is $|S_n(d')\setminus S_n(d'')| = O(\lambda^{-1}s')= o(n)$, so they make up for at most a $\lambda^{-1}$-proportion of  vertices in $S_n(d')$. 
Since these are the vertices of smallest degree in $S_n(d')$, the number of edges incident to them is at most $O(\lambda^{-1} m)$. 
Hence, all but $O(\lambda^{-1}m)$ edges are incident to two vertices of degree at least $d''$. Since any such edge belongs to $\cC_1$ with probability at least $1-2\lambda^{-1}$, 
Markov's inequality implies that, with probability at least $1-\lambda^{-1/2}$, all but $O(\lambda^{-1/2}m)$ edges are contained in $\cC_1$, yielding the lemma.
\end{proof}

Define $\cA_3$ as the following event: for every set $U\subseteq V(\cC_1)$ satisfying $|U|\leq d_1+\sqrt{d_2d}$, the number of red edges of $\mathbf{G}(\cD_n)$ with both endpoints in $V(\cC_1)\setminus U$ is at least $(1-1/c)m/4$.

\begin{corollary}\label{cor dom 2}
We have $\Pr(\cA_3)=1-o(1)$.
\end{corollary}
\begin{proof}
On the one hand, Chernoff's bound implies that there must be at least $(1-p_1)m/3 > (1-1/c)m/3$ red edges in total.
On the other hand, by Lemma~\ref{cor dom}, a.a.s.\ there are $(1-o(1))m/2$ edges with both endpoints in $V(\cC_1)$. 
Since $d_2=o(d)$ and $d_1=O(d_2)$ (assumption made in the beginning of this subsection), for every set $U\subseteq V(\cC_1)$ satisfying $|U|\leq d_1+\sqrt{d_2d}$, there are at most $|U|n \le d_1 n+ \sqrt{d_2 d}\, n = o(m)$ edges, incident to $U$.
Since $(1-1/c)m/3 - o(m) \ge (1-1/c)m/4$, the corollary follows.
\end{proof}

\begin{lemma}\label{lem last}
We have $\Pr(\overline{\cA_1\cup \cA_2})=o(1)$.
\end{lemma}
\begin{proof}
For $i\in \{1,2,3\}$, recall that $N^i(v)$ denotes the set of neighbors of $v$ of type $i$, and consider the random variables
$$
X_i= \sum_{v\in S_{iso}} |N^i(v)|. 
$$
Let $k_0\coloneqq d_2 \delta n/72$ and define the event 
$$
\cB=\{X_3> k_0\}.
$$
Under $\cA_0$, we have $\overline{\cA_1\cup \cA_2}\subseteq \cB$: indeed, observe that under $\cA_0$, if $\overline{\cA_1\cup \cA_2}$ holds, then there has to be at least a third of the vertices $v\in S_{iso}$ with $|N^3(v)|> d_2/3$ which gives $X_3> d_2|S_{iso}|/9\geq k_0$. We will show that, under the event $\cA_0\cap \cA_3$, the probability that $\cB$ holds tends to $0$. This completes the proof of the lemma since $\cA_0$ and $\cA_3$ hold a.a.s.

For every $t, k \ge 0$, define $\cF_{t,k}$ as the event that $\cA_0\cap \cA_3$ holds, $X_2+X_3=t$ and $X_3=k$. We have
\begin{align}
    \cB\cap \cA_0\cap \cA_3  = \bigcup_{t\ge k > k_0} \cF_{t,k}.
\end{align}
For $t$ fixed, in the switching argument below, it will be convenient to split the events $\cF_{t,k}$ into disjoint families of size $n$. 
Consider $\cG_{t,\ell} = \cup_{i=0}^{n-1} \cF_{t,n\ell+i}$. 
Letting $\ell_0= \lfloor k_0/n\rfloor$, it suffices to bound the probability of
\begin{align}
    \cB\cap \cA_0\cap \cA_3 = \bigcup_{\substack{t\geq 0 \\ \ell> \ell_0}} \cG_{t,\ell}.
\end{align}

Before going into the technical details, let us give an intuitive explanation of the proof. 
To estimate the probability of $\cB\cap \cA_0\cap \cA_3$, we bound from above the probabilities of $\cG_{t,\ell}$ in a recursive fashion by using switchings. 
Roughly speaking, the events $(\cG_{t,\ell})_{t,\ell}$ could be seen as layers that are consecutively peeled from the event $\cB\cap \cA_0\cap \cA_3$. 
Note that, while $\cF_{t,k}$ seem to be natural candidates to define the layers (in a simpler way), the possibility that the value of $X_3$ can change significantly after a single switching makes the recursive relations long and unpleasant.
Instead, the probability of $\cG_{t,\ell}$ is dominated by using switchings with a closely related event $\cH_{t,\ell}$ defined as union of only six events preceding $\cG_{t,\ell}$ in the lexicographic order.
To do this, we use switchings again. 
We first study ``backward'' switchings. 
Let $G\in \cG_{t,\ell}$ with $\ell > \ell_0$. 
Recall that $\cC_1$ is the largest (blue) component of $G_1$ after the first percolation stage. 
We will switch red edges of the form $vz$, where $v\in S_{iso}$ and $z\in N^3(v)$, with red edges $xy$ for which both endpoints $x$ and $y$ belong to $\cC_1$, thus obtaining the edges $vx$ and $yz$. 
Note that since the edge $xy$ is red, $x$ and $y$ still belong to $\cC_1$ after the switching. 
Now, observe that after this switching, $X_2+X_3$ decreases by 1 or 2: 
indeed, $z \in N^3(v)$, and therefore $|N^3(v)|$ decreases by $1$. It might also happen that $z \in S_{iso}$ and $v \in N^2(z) \cup N^3(z)$, 
in which case $|N^2(z)|+|N^3(z)|$ also decreases by $1$. 
For any other vertex $w$ in $S_{iso}$ being incident via (red) edges to either $v$ or $z$, note that $v$ or $z$ might switch from $N^3(w)$ to $N^2(w)$ but this does not affect $X_2+X_3$. 
By the same argument, $X_3$ alone might decrease by at most $2|S_{iso}|$ (this is the case if $v$ or $z$ have exactly $\sqrt{d_2 d}$ (red) neighbors in $\cC_1$ after the switching), 
see Figure~\ref{fig switch}. Thus, the switched graph is in 
\begin{equation*}
    \mathcal H_{t, \ell} :=\cG_{t-1,\ell}\cup \cG_{t-1,\ell-1}\cup \cG_{t-1,\ell-2}\cup \cG_{t-2,\ell}\cup \cG_{t-2,\ell-1}\cup \cG_{t-2,\ell-2}.
\end{equation*}

\begin{figure}
\centering
\begin{tikzpicture}[scale=1.5,line cap=round,line join=round,x=1cm,y=1cm]
\clip(-6,-2.2) rectangle (3,2.2);
\draw [rotate around={0:(-1.5,1)},line width=1pt] (-1.5,1) ellipse (1.825140769936443cm and 1.039778260055571cm);
\draw [color = red, line width=0.8pt] (-1,-1.5)-- (-1,-2);
\draw [color = red, line width=0.8pt] (-2,-1.5)-- (-1,-2);
\draw [color = red, line width=0.8pt] (-1,-1.5)-- (-2,-2);
\draw [color = red, line width=0.8pt] (-2,-1.5)-- (-3,-1);
\draw [color = red, line width=0.8pt] (-2,-1.5)-- (-3,-1.5);
\draw [color = red, line width=0.8pt] (-2,-1.5)-- (-2,-2);
\draw [color = red, line width=0.8pt] (-1,-1.5)-- (0,-1.5);

\draw [color = red, line width=0.8pt] (-2,0.5)-- (-1,0.5);
\draw [color = red, line width=0.8pt] (-2,-1.5)-- (-1,-1.5);
\draw [color = red, line width=0.8pt] (-2,-1.5)-- (-3,1);
\draw [color = red, line width=0.8pt] (-1,-1.5)-- (0,1);
\draw [color=red,line width=0.4pt,dash pattern=on 2pt off 2pt] (-2,0.5)-- (-2,-1.5);
\draw [color=red,line width=0.4pt,dash pattern=on 2pt off 2pt] (-1,0.5)-- (-1,-1.5);
\draw [color=red,line width=0.8pt] (-1,-1.5)-- (-0.5,1);
\begin{scriptsize}
\draw [fill=black] (-2,0.5) circle (1.5pt);
\draw[color=black] (-2,0.65) node {\large{$x$}};
\draw [fill=black] (-1,0.5) circle (1.5pt);
\draw[color=black] (-1,0.65) node {\large{$y$}};
\draw[color=black] (0.5,0.65) node {\Large{$\cC_1$}};
\draw [fill=black] (-2,-1.5) circle (1.5pt);
\draw[color=black] (-1.9,-1.35) node {\large{$v$}};
\draw [fill=black] (-1,-1.5) circle (1.5pt);
\draw[color=black] (-1.1,-1.35) node {\large{$z$}};
\draw[color=black] (-2.2,-2) node {\large{$w$}};
\draw [fill=black] (-3,-1) circle (1.5pt);
\draw [fill=black] (-3,-1.5) circle (1.5pt);
\draw [fill=black] (-3,1) circle (1.5pt);
\draw [fill=black] (-2,-2) circle (1.5pt);
\draw [fill=black] (0,-1.5) circle (1.5pt);
\draw [fill=black] (0,1) circle (1.5pt);
\draw [fill=black] (-1,-2) circle (1.5pt);
\draw [fill=black] (-0.5,1) circle (1.5pt);
\end{scriptsize}
\end{tikzpicture}
\caption{The switching of $vz$ and $xy$ with $vx$ and $zy$. First, $v$ and $z$ may move from $S_{iso}^2\cup S_{iso}^3$ to $S_{iso}^1$ (note that $z$ may be outside $S_{iso}$ in general while $v\in S_{iso}$ by definition), and therefore for every vertex $w\in S_{iso}$, the edges $wv$ and $wz$ may contribute to $N_2(w)$ instead of $N_3(w)$ after the switching. Apart from that, $z$ loses one edge towards $S_{iso}$ and $v$ loses at most one edge towards $S_{iso}$. Hence, $X_3$ changes by at most $e(v, S_{iso}\setminus v) + e(z, S_{iso}\setminus z)\le 2|S_{iso}|$ but $X_2+X_3$ changes by 1 or 2.}
\label{fig switch}
\end{figure}

There are at least $\ell n$ choices for $v z$. 
Since by definition $\cG_{t,\ell}\subseteq \cA_3$, $d(v)\leq d_1$ and $z$ has at most $\sqrt{d_2 d}$ neighbors in $\cC_1$, 
by Corollary~\ref{cor dom 2} applied with $(N(v) \cup N(z)) \cap V(\cC_1)$ there are $\Omega(m)$ choices for the (red) edge $xy$, 
which ensures $\Omega(\ell n m)$ possible switchings from $\cG_{t,\ell}$ to  $\mathcal{H}_{t, \ell}$.

We now look at ``forward'' switchings. Let $G\in \cH_{t,\ell}$. To obtain a graph in $\cG_{t,\ell}$ we need to switch two edges, $vx$ and $yz$, such that $v\in S_{iso}$, $z \in V \setminus \cC_1$ with less than $\sqrt{d_2d}+1$ neighbors in $\cC_1$ and $x,y \in \cC_1$, to obtain the edges $vz$ and $xy$. There are clearly at most $n^2$ choices for $v$ and $z$, and, given $v$ and $z$, there are at most $d_1(\sqrt{d_2d}+1)$ ways to complete the switching (recall that vertices in $S_{iso}$ have degree at most $d_1$). So the total number of switchings is at most $n^2 d_1 (\sqrt{d_2d}+1)$. 

For $\ell \ge \ell_0$, we double-count the number of switchings between $\cG_{t,\ell}$ and $\mathcal H_{t,\ell}$. Recall that $m\geq (\delta' n-o(n)) d$  and $d_1=O(d_2)$, which in particular implies $d_1=o(\sqrt{d_2 d})$. It follows that 
$$
|\cG_{t,\ell}| \leq \frac{n^2 d_1 (\sqrt{d_2 d}+1)}{\Omega(\ell nm)}\mathcal |\cH_{t,\ell}| = o(|\mathcal H_{t,\ell}|).
$$
Therefore, for every $t\in \mathbb N$,
\begin{align*}
    |\bigcup_{\substack{t\geq 0\\ \ell\geq \ell_0}} \cG_{t,\ell}| =\sum_{\substack{t\geq 0\\ \ell\geq \ell_0}} |\cG_{t,\ell}| = o(1)\sum_{\substack{t\geq 0\\ \ell\ge \ell_0}} |\mathcal{H}_{t,\ell}| =o(1)\sum_{\substack{t\geq 0\\ \ell\ge \ell_0-2}} |\mathcal{G}_{t,\ell}|,
\end{align*}
since each of $(\cG_{t,\ell})_{t\ge 0,\ell\ge \ell_0}$ appears in at most $6$ terms of the sum $\sum_{t\geq 0, \ell\ge \ell_0} |\mathcal{H}_{t,\ell}|$.

We conclude that
$$
\Pr\big(\bigcup_{\substack{t \geq 0\\ \ell\geq \ell_0}} \cG_{t,\ell}\big)
= o\big(\Pr\big(\bigcup_{\substack{t \geq 0 \\ \ell\geq \ell_0-2}} \cG_{k_0,\ell}\big)\big)= o(1),
$$
and the lemma holds. 
\end{proof}

\begin{proof}[Proof of Theorem~\ref{thm 2}]
Let $\cA=\{L_1(\mathbf{G}_2) - L_1(\mathbf{G}_1)\geq c_3\delta n\}$, where $c_3$ is the constant appearing in Corollary~\ref{trivial cor}.
By Corollary~\ref{trivial cor} and Lemma~\ref{lem last}, we have
\begin{align*}
    \Pr(\cA) \geq \Pr(\cA\cap (\cA_1\cup \cA_2))\geq \Pr(\cA_1\cup \cA_2)-o(1) = 1-o(1).
\end{align*}
\end{proof}

\begin{proof}[Proof of Corollary~\ref{thm 3}]
Corollary~\ref{thm 3} is a direct application of Theorem~\ref{thm 2}, with $ (1+\varepsilon)^{-1} d_2$ and $(1-\varepsilon)^{-1} d_1$ instead of $d_2$ and $d_1$, respectively.
\end{proof}

\subsubsection{\texorpdfstring{Remarks on Theorem~\ref{thm 2} and Corollary~\ref{thm 3}}{}}\label{sec:Remarks}

In this section we show that the conditions $\delta' > 0$ and $\delta > 0$ in two of our main results, Theorem~\ref{thm 2} and Corollary~\ref{thm 3}, are needed.

\vspace{1em}

\begin{remark}\label{rem:1}
In both Theorem~\ref{thm 2} and Corollary~\ref{thm 3}, the condition $\delta'>0$ is necessary. In the setting of Corollary~\ref{thm 3}, the next example shows the existence of degree sequences for which a linear size component does not exist a.a.s. even for $p_2$-percolation if $\delta' =0$. Fix $\varepsilon,\delta > 0$, $d_2 = d_2(n) = \log n-\log \log n$ and $d_1 = d_1(n)= \log n + \log \log n$. Let $\cD_n$ be the sequence formed by $(1-\delta) n$ vertices of degree $\delta \log n$ and $\delta n$ vertices of degree $\log n$. We will sketch the argument using the configuration model $\mathbf{CM}(\cD_n)$. For a degree sequence $\mathcal D_n$, define the \emph{$p$-thinned version} of $\cD_n$, denoted by $(\cD_n)_p$, as the random degree sequence satisfying that $\forall v\in V, d(v)_p \sim \text{Bin}(d(v),p)$ where $(d(v)_p)_{v\in V}$ are independent random variables conditioned on their sum being even. It is well-known that \emph{for any} degree sequence, $\mathbf{CM}(\cD_n)_p$ is distributed as $\mathbf{CM}(\cD_n)_p$ (see e.g.~\cite[Lemma 3.2]{fountoulakis2007percolation}). As $n$ grows, $d(v)_{p_2}$ converges to $\text{Po}(d(v){p_2})$ in distribution, and $(\cD_n)_{p_2}$ converges in distribution to the random degree sequence $\hat \cD_n$ composed of $(1-\delta)n$ vertices $u\in V: \hat{d}(u)\sim \mathrm{Po}(\delta(1+\eps))$ and $\delta n$ vertices $u\in V: \hat{d}(u)\sim \mathrm{Po}(1+\eps)$ for some $\eps = \eps(n) = o(1)$, conditionally on $\sum_{u\in V} \hat{d}(u)$ being even. It is easy to check that one may choose sufficiently small $\delta > 0$ so that a.s.
$$
\lim_{n\to \infty} \dfrac{\sum_{v\in V} \hat{d}(v)}{n} \text{ exists and }\lim_{n\to \infty} \dfrac{\sum_{v\in V} \hat d(v)(\hat d(v)-2)}{\sum_{v\in V} \hat d(v)} \text{ exists and is negative}.
$$
Then, by Theorem 2 in~\cite{bollobas2015old} the probability that there exists a component of linear order is at most $\exp(-cn)$ for some $c>0$. Moreover, by choosing $\delta$ sufficiently small one may conclude by Corollary 11.8 in~\cite{FK} that the probability that $\mathbf{CM}(\cD_n)$ is simple is at least $\exp(-O(\frac{1}{m}\sum d(v)^2)) \geq  \exp(-c' n)$ for some $0< c' < c$ (recall that $m = \sum d(v)$.) Therefore, the probability that $L_1(\mathbf{G}(\cD_n)_{p_2})=o(n)$ is at least $1-\exp(-(c-c')n)= 1-o(1)$. A similar example may be constructed for Theorem~\ref{thm 2} as well by choosing $d_2 = d_2(n)=\log n$, $d_1 = d_1(n)=3 \log n$, and $\cD_n$ 
being the sequence formed by $(1-\delta) n$ vertices of degree $\delta \log n$ and $\delta n$ vertices of degree $2 \log n$ for sufficiently small $\delta > 0$.
\end{remark}

\begin{remark}\label{rem:2}
In Theorem~\ref{thm 2} and Corollary~\ref{thm 3}, the condition $\delta > 0$ is also necessary. For example, in the setup of Theorem~\ref{thm 2}, let $\cD_n$ be the constant degree sequence where $d(v)=n-1$ for any $v\in V$. Then, $\mathbf{G}(\cD_n)$ is the complete graph and $\mathbf{G}(\cD_n)_p$ is an Erd\H os-R\'enyi random graph with parameters $n$ and $p$. Let $d_1(n)=\log n$ and $d_2(n)=2\log n$, so $\delta=0$. Since $p_1, p_2=\Theta(1/\log n)$, both $\mathbf{G}(\cD_n)_{p_1}$ and $\mathbf{G}(\cD_n)_{p_2}$ are a.a.s. connected, and in particular
\begin{equation*}
    \mathbb P\left[L_1(\mathbf{G}(\cD_n)_{p_2}) = L_1(\mathbf{G}(\cD_n)_{p_1})\right] = 1-o(1).
\end{equation*}
\end{remark}

\section{Degree sequences with multiple jumps}\label{sec:exm2}

In this section, we show the existence of degree sequences such that the percolated random graph undergoes an unbounded number of phase transitions in the order of the largest connected component. The two examples below are of different nature. Firstly, we show that the proportion of vertices in a linear connected component can experience many jumps. Secondly, we show that even the asymptotic order of the largest component can abruptly change many times. This shows that determining the asymptotic size of the largest component of percolated dense random graphs is a rich and complex problem.
      
\subsection{Jumps in the leading constant}\label{sec:exm1}

Consider the degree sequence $\cD_n$ that has $n/2$ vertices of degree $n^{1/2}$, $n/4$ vertices of degree $n^{3/4}$, $n/8$ vertices of degree $n^{7/8}$, and so on. 
In particular, the average degree of $\mathbf{G}(\cD_n)$ is $n^{1-o(1)}$ and most of the edges have both endpoints in vertices of degree $n^{1-o(1)}$. 
From an intuitive point of view, a typical graph with degree sequence $\cD_n$ looks like an onion: it contains many nested layers with different degree profiles, the most external being the vertices of degree $\sqrt{n}$ and the most internal ones being the vertices of highest degrees.  

For any $\alpha\in (0,1]$, set $p_{\alpha} = p_{\alpha}(n) = n^{\alpha-1}$. Using that the quantity $|S_n(d)|$ jumps at $d = n^{1-1/2^k}$ for all $k\ge 1$, a direct application of Theorem~\ref{thm 1} shows that, for every $k\ge 1$,
\begin{itemize}
    \item if $\alpha > 2^{-k}$, then a.a.s.\ $L_1(\mathbf{G}(\cD_n)_{p_{\alpha}})\ge n/2^{k-1}-o(n)$, and
    \item if $\alpha < 2^{-k}$, then a.a.s.\ $L_1(\mathbf{G}(\cD_n)_{p_{\alpha}})\le n/2^k+o(n)$. 
\end{itemize}
As a result, the step function $f:\alpha\in (0,1]\mapsto 2^{\lceil \log_2(\alpha)\rceil}\in (0,1]$ depicted in Figure~\ref{fig 2} captures the evolution of the largest component in the sense that, for every $\alpha\in (0,1]$, $L_1(\mathbf{G}(\cD_n)_{p_{\alpha}}) = f(\alpha) n+o(n)$.

Interestingly, there is no threshold for the property of having a linear order component: for any $\alpha\in (0,1]$, a.a.s.\ there is a linear order component in the graph percolated with probability $p_{\alpha}$ but there is a.a.s.\ no such component if $p = n^{-1+o(1)}$.

\begin{figure}[ht]
\begin{center}
\includegraphics[width=0.6\textwidth]{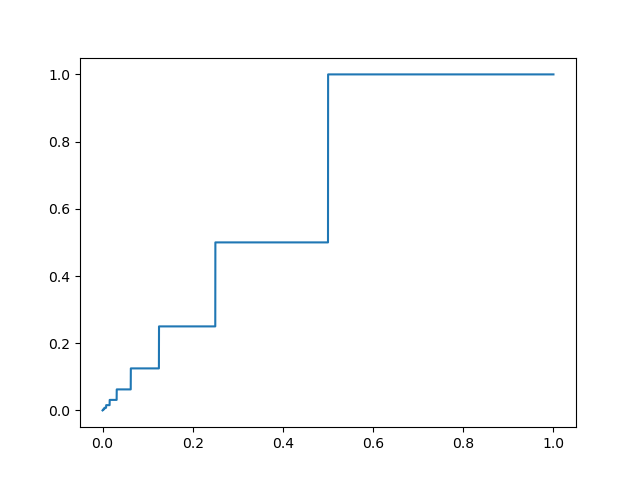}
\end{center}
\caption{The function $f(\alpha)$ for $\alpha\in(0,1]$.}
\label{fig 2}
\end{figure}

      \subsection{Jumps in the order: the proof of Theorem~\ref{thm example}}

In this section we present a degree sequence for which there are many jumps in the asymptotic order of the largest component. We start with the following proposition:

\begin{proposition}\label{thm 7.1}
Fix $d = d(n)$ and $p = p(n) \in [0,1)$ satisfying $dp\gg 1$. Let $\mathbf{H}=\mathbf{H}(\cD_n)$ be a random graph satisfying that the degrees of the vertices of the graphs $\mathbf{H}[S_{\mathbf{H}}(d)], \mathbf{H}[V\setminus S_{\mathbf{H}}(d)]$ and $\mathbf{H}[S_{\mathbf{H}}(d), V\setminus S_{\mathbf{H}}(d)]$ are prescribed, and the three graphs are uniform with respect to the given degrees. 
Denote $S(d) = S_{\mathbf{H}}(d)$ for ease of writing and suppose that there is a function $\lambda:\mathbb N\to \mathbb R^+$ with $\lambda(n)$ tending to $\infty$ as $n \to \infty$ such that:
\begin{enumerate}
    \item[i)]\label{pt 1} Asymptotically almost all vertices $v\in S(d)$ satisfy $\lambda^2/p\le \lambda\, d_{\mathbf{H}[S(d)]}(v)\le d_{\mathbf{H}}(v)$,
    \item[ii)]\label{pt 2} In $\mathbf{H}[S(d), V\setminus S(d)]$, all vertices in $V \setminus S(d)$ have degree at most $(\lambda p)^{-1}$.
\end{enumerate}
Then, the largest component in $\mathbf{H}_p$ has size at least $(1-o(1)) pd|S(d)|$ a.a.s.
\end{proposition}
\begin{proof}
All subsequent claims in the proof hold only a.a.s. First, by assumption i) almost all vertices in $S(d)$ have degree at least $\lambda/p$ in $\mathbf{H}[S(d)]$ (so in particular $|S(d)|\ge \lambda/p$). Hence, by Theorem~\ref{thm 1}~(b) (applied with $\mathbf{H}[S(d)]$ instead of $\mathbf{G}(\mathcal D_n)$ and $\lfloor \lambda/p\rfloor$ instead of $d$), the largest component of $\mathbf{H}[S(d)]_p$ has size $(1-o(1))|S(d)|$; let $V_0$ be the corresponding vertex set of this component.  Let $k_0 = \lceil (\lambda p)^{-1}\rceil$. For $k\leq k_0$, let $n_k$ be the number the vertices in $V\setminus S(d)$ that have exactly $k$ neighbors in $V_0$. Recall that for every $v\in S(d)$ we have $d(v)\ge d$ and, for most of them, $ e(v,S(d))\le d(v)/\lambda=o(d(v))$.
Therefore
$$
(1-o(1))d|S(d)|\leq e(V_0, V\setminus S(d))= \sum_{k=1}^{k_0} k n_k.
$$
Denote by $U$ the size of the set of vertices in $V\setminus S(d)$ adjacent to $V_0$ after percolation. Then we have
$$
\bE[|U|]= \sum_{k=1}^{k_0} n_k(1-(1-p)^k) \sim p \sum_{k=1}^{k_0} k n_k \geq (1-o(1))pd|S(d)|,
$$
where we used that $(1-(1-p)^k)\sim kp$ since $kp\leq k_0p=o(1)$. As $|S(d)|\geq \lambda/p \to \infty$ as $n\to\infty$, a simple application of Chernoff's inequality shows that $|U|\sim \bE[|U|]$.
\end{proof}

We will use the previous proposition to prove the following statement on the existence of a large component.
\begin{proposition}\label{prop:gen2}
    Let $0<\beta < \alpha<1$ with $\alpha + \beta < 1 < \alpha + 2\beta$. Let $\omega = \omega(n)$, $\hat{\omega} = \hat{\omega}(n)$ be two functions tending to $\infty$ sufficiently slowly as $n\to \infty$ and $\hat\omega(n)\ll \omega(n)$. For $d = n^{\beta}$ and $p = \hat{\omega} n^{1-\alpha-2\beta}$ suppose that the following conditions are satisfied:
\begin{enumerate}[i)]
    \item\label{(i)} $|S_n(d)| \sim n^{\alpha}$,
    \item\label{(ii)} $d(S_n(d))\sim n^{\alpha + \beta}$,
    \item\label{(iii)} $\Delta(V\setminus S_n(d))\le (\omega p)^{-1}$,
    \item\label{(iv)} $d(V\setminus S_n(d)) \sim n$,
    \item\label{(v)} $\Delta(S_n(d))\le n^{2-\alpha-2\beta}$.
\end{enumerate}
Then, the largest component in $\mathbf{G}(\cD_n)_p$ has size at least $n^{1-\beta}$.
\end{proposition}

Before proving the proposition we will need two auxiliary lemmas. Let us define
\begin{align*}
    & A = \{v\in S_n(d)\setminus S_n(2d):  e(v, S_n(d))\le  \tfrac{1}{8}n^{\alpha+2\beta-1}\},\\
    & B = \{v\in S_n(d): e(v, S_n(d))\ge 2 n^{\alpha+\beta-1} d(v)\}.
\end{align*}

In the sequel, we denote $\mathbf{G} = \mathbf{G}(\cD_n)$.

\begin{lemma}\label{lem 3.2}
A.a.s.  $A = \emptyset$.
\end{lemma}
\begin{proof}
Fix a vertex $v$ of degree $d(v)\in [n^{\beta}, 2 n^{\beta}]$ and let $\gamma=\alpha + 2\beta - 1$. We use the switching method to show the statement of the lemma. Denote by $\cF_i$ the event that $e(v,S_n(d)\setminus v) = i$. Suppose that $i\le \tfrac{1}{4} n^{\gamma}$. The number of switchings from $\cF_i$ to $\cF_{i-1}$ is bounded from above by $2 i\cdot d(V(\mathbf{G})) \le (\tfrac{1}{2}+o(1)) n^{\gamma+1}$, by \eqref{(ii)} and \eqref{(iv)}.
On the other hand, to construct a switching from $\cF_{i-1}$ to $\cF_i$ it suffices to choose  a vertex $u\in N_{\mathbf{G}}[v]\setminus S_n(d)$ and an edge $xy$ with $x\in S_{n}(d)\setminus N_{\mathbf{G}}[v]$ and $y\in V\setminus N_{\mathbf{G}}[u]$ to obtain the edges $vx$ and $uy$ after the switching. By \eqref{(i)} and~\eqref{(iii)}, and since $\beta<\alpha$, the number of choices is at least
\begin{align*}
\sum_{u\in N_{\mathbf{G}}[v]\setminus S_n(d)} d(S_n(d)\setminus N_{\mathbf{G}}[v], V\setminus N_{\mathbf{G}}[u]) &\geq (d - (i-1)) |S_n(d)\setminus N_{\mathbf{G}}[v]|(d - \Delta(V\setminus S_n(d)) - 1)\\
&~\sim n^{\alpha+2\beta}=n^{\gamma+1}.
\end{align*}

We conclude that 
\begin{equation*}
    |\cF_{i-1}|\le (\tfrac{1}{2} + o(1))|\cF_i|, 
\end{equation*}
and
\begin{equation*}
    \mathbb P(v\in A)\le \dfrac{|\bigcup_{i \le \frac{1}{8}n^{\gamma}} \cF_{\lceil \frac{1}{8} n^{\gamma}\rceil}|}{|\cF_{\lfloor\frac{1}{4} n^{\gamma}\rfloor}|} \leq (\tfrac{1}{2} + o(1))^{\frac{1}{8}n^{\gamma}}.
\end{equation*}
Thus, by a union bound, a.a.s.\ $A = \emptyset$.
\end{proof}

\begin{lemma}\label{lem 3.3}
A.a.s. $B = \emptyset$.
\end{lemma}
\begin{proof}
Fix a vertex $v \in S_n(d)$ of degree $d(v)\geq n^{\beta}$. Denote by $\cF_i$ the event that $e(v, S_n(d)) = i$. Suppose that $i\ge n^{\alpha+\beta-1} d(v)$. The number of switchings from $\cF_{i-1}$ to $\cF_i$ is bounded from above by $d(v)\cdot d(S_n(d)) \sim n^{\alpha+\beta} d(v)$, by \eqref{(ii)}.

On the other hand, for each $u\in N_{\mathbf{G}}(v)\cap S_n(d)$ we may switch $vu$ with any edge $wz$ provided that $w,z\in V\setminus (S_n(d)\cup N_{\mathbf{G}}[u]\cup N_{\mathbf{G}}[v])$. Thus, for a fixed $u$, by \eqref{(ii)}, \eqref{(iii)}, \eqref{(iv)} and \eqref{(v)} the number of choices for $wz$ is at least
$$
(1-o(1)) n - (1+o(1)) n^{\alpha+\beta} - 2\Delta(S_n(d))\Delta(V\setminus S_n(d))\sim n,
$$
and, since each choice gives rise to two switchings, the total number of switchings from $\cF_i$ to $\cF_{i-1}$ is at least
$$
2i\cdot (1-o(1))n \geq (2-o(1)) n^{\alpha+\beta} d(v).
$$

We conclude that
\begin{equation*}
   |\cF_{i-1}|\geq \left(2-o(1)\right) |\cF_i|,
\end{equation*}
and, as in the proof of the previous lemma,
$$
\Pr(v\in B)\leq  \left(2 - o(1)\right)^{-n^{\alpha+\beta-1}d(v)} \le \left(2 - o(1)\right)^{-n^{\alpha+2\beta-1}}.
$$ 
Thus, by a union bound, a.a.s.\ $B = \emptyset$.
\end{proof}

\begin{proof}[Proof of Proposition~\ref{prop:gen2}]
Fix $\lambda=\hat\omega/8$. We will show that the (random) degree sequences of the graphs $\mathbf{G}[V\setminus S_n(d)]$, $\mathbf{G}[S_n(d), V\setminus S_n(d)]$ and $\mathbf{G}[S_n(d)]$ a.a.s.\ satisfy the two conditions of Proposition~\ref{thm 7.1} with the given $d$ and $p$. First, \eqref{(iii)} implies the second condition as $\lambda\leq \omega$. To show that the first condition holds, we use Lemmas~\ref{lem 3.2} and~\ref{lem 3.3}. By \eqref{(i)} and \eqref{(ii)} we have that a $(1-o(1))$-proportion of the vertices of $S_n(d)$ have degree $(1+o(1))n^\beta$, and by Lemma~\ref{lem 3.2}, a.a.s.\ all such vertices satisfy 
$$
\lambda e(v,S_n(d))\geq \frac{\lambda}{8} n^{\alpha+2\beta-1} =\frac{\lambda^2}{p}. 
$$
Moreover, by Lemma~\ref{lem 3.3}, every vertex $v$ in $S_n(d)$ satisfies
$$
\lambda e(v,S_n(d))\leq 2 \lambda n^{\alpha+\beta-1}d(v) \leq d(v). 
$$
since $\lambda$ increases sufficiently slowly to infinity, while $\alpha+\beta<1$. 

By applying Proposition~\ref{thm 7.1} we obtain the result.
\end{proof}

With the results previously developed, we are able to provide a proof of Theorem~\ref{thm example}.

\begin{proof}[Proof of Theorem~\ref{thm example}]
We choose  $\alpha_0 = 0$, and $\alpha_i = 1 - \frac{11}{8\cdot 5^{i+1}}$ and $\beta_i = \delta_i = \frac{1}{5^{i+1}}$ for every $i\in [k]$. Note that $(\alpha_i)_i$ is increasing whereas $(\beta_i)_i$ is decreasing.

For Part~(b), we apply Proposition~\ref{prop:gen2} with $\alpha = \alpha_i, \beta = \beta_i$ and suitable chosen $\hat\omega$. Let us check that all conditions are satisfied. Condition \eqref{(i)} is trivial, \eqref{(ii)} is satisfied since
$$
(n^{\alpha_i}-n^{\alpha_{i-1}})n^{\beta_i} \leq d(S_n(d))\leq  \sum_{j=1}^i n^{\alpha_j+\beta_j}\sim n^{\alpha_i+\beta_i},
$$
\eqref{(iii)} is satisfied since $\Delta(V\setminus S_n(d)) = n^{\beta_{i+1}}\ll n^{\alpha+2\beta-1}$, \eqref{(iv)} is satisfied since the total number of edges incident to $V \setminus S_n(d)$ is at least  $n-n^{\alpha_k}+1\sim n$, and at most $n + \sum_{j=1}^k n^{\alpha_j+\beta_j} \sim n$, and~\eqref{(v)} is satisfied since for all $i\in [k]$ one has $2-2\beta - \alpha = 1 - (8 \cdot 5^i)^{-1} \ge 39/40\geq 1/25 =\beta_1\geq \beta$. We conclude by Proposition~\ref{thm 7.1} that $L_1(\mathbf{G}_{p})\ge n^{1-\beta}$.

Let us now prove Part~(a). Fix $i \in [2,k]$ and set $d_i = n^{\beta_i}$. By Lemma~\ref{lem 3.3} we may assume that $B = \emptyset$. Let $S^i=S_n(d_i)$. For any $v\in V$, we will show that the total order of the components of $\mathbf{G}_p$ containing $S^{i-1}\cup \{v\}$ is $O(n^{1-\beta_{i}-\delta_i})$ with probability $o(1/n)$. A union bound over all choices of $v$ will allow us to conclude.

To bound the order of the components, we will couple its exploration process with a subcritical 2-type branching process that stochastically dominates it. The first type of offspring will correspond to vertices in $T_1=S^i\setminus S^{i-1}$ (all having degree $n^{\beta_i}$ in $\mathbf{G}$), and the second type to vertices in $T_2=V\setminus S^i$ (all having degree at most $n^{\beta_{i+1}}$). Note that the vertices in $S^{i-1}$ do not appear throughout the exploration process (below, we show that the total number of edges incident to $S^{i-1}$ is ``small'' a.a.s., which will allow us to ignore them).
 
Let $\gamma_i=\alpha_i+2\beta_i-1 = \tfrac{1}{8\cdot 5^{i}}$ and $\omega = (n^{\gamma_i} p)^{-1}$. Then, under the a.a.s.\ event that $B=\emptyset$, for any vertex $x\in T_1$, the number of children of $x$ in $T_1$ in the exploration process is stochastically dominated by a $\text{Bin}(2n^{\gamma_i},p)$.

Consider thus the following 2-type branching process $(X_t)_{t\ge 0}=(X_{t,1},X_{t,2})_{t\ge 0}$ with types $1$ and $2$ having offspring distributions 
\begin{align*}
    \xi_{11}\sim \text{Bin}(2n^{\gamma_i},p),\;
    \xi_{12}\sim \text{Bin}(d_i,p),\;
    \xi_{21}\sim \text{Bin}(d_{i+1},p),\;\text{and }
    \xi_{22}\sim \text{Bin}(d_{i+1},p).
\end{align*}
Note that, if $X_{0,1}=1$ and  $X_{0,2}=0$, then $(X_t)_{t\ge 0}$ stochastically dominates the size of the exploration process starting at a vertex in $T_1$, whilst if  $X_{0,1}=0$ and $X_{0,2}=1$, then it dominates the process starting at a vertex in $T_2$. By construction, the process $(X_t)_{t\ge 0}=(X_{t,1},X_{t,2})_{t\ge 0}$ has mean matrix
$$
\begin{pmatrix}
2n^{\gamma_i} p & d_i p\\
d_{i+1}p & d_{i+1} p\\
\end{pmatrix}
\leq 
\begin{pmatrix}
2/\omega & n^{\frac{3}{8\cdot 5^{i+1}}}\\
n^{-\frac{17}{40\cdot 5^{i+1}}} & n^{-\frac{17}{40\cdot 5^{i+1}}}\\
\end{pmatrix}=:M.
$$

Both the trace and the determinant of $M$ are $o(1)$, and thus its largest eigenvalue is $o(1)$. By standard results in multi-type branching processes (see e.g. Theorem VIII.3.2 in~\cite{athreya1972}), $(X_t)_{t\ge 0}$ is subcritical.

We prove the following result on the total progeny of the branching process.

\begin{claim}\label{claim X}
We have $\sum_{t\ge 0} (X_{t,1}+X_{t,2}) =O( (\log n)^4 \mathbb{E}[\xi_{12}])$ with probability $1 - o(n^{-2})$.
\end{claim}
\begin{proof}
We will show the claim for $X_{0,1}=1$ and $X_{0,2}=0$. We may assume that this is the case since we can always artificially add a parent of type $1$ to the root.

We will bound the total progeny by the product of the total progenies of two one-type branching processes (we will make this precise below). Consider first the one-type branching process $(Y_t)_{t\geq 0}$ where the root has offspring distribution $\xi_{12}$ and all other elements have offspring distribution $\xi_{22}$. Define the events
\begin{equation*}
\cA= \{Y_1 \le 2 \mathbb{E}[\xi_{12}]\} \text{ and } \cB = \Big\{\sum_{t\geq 1} Y_t\le 2(\log n)^2  \mathbb{E}[\xi_{12}]\Big\}.
\end{equation*}
By Chernoff's inequality (Lemma~\ref{chernoff}),
\begin{equation*}
\mathbb P(\overline{\cA}) = \mathbb P(\mathrm{Bin}(d_i,p) > 2 d_i p) = \exp(-\Omega(d_i p)) = o(n^{-3}). 
\end{equation*}
Writing $Y_{i,t}$ for the number of descendants of the $i$-th child of the root at level $t\ge 1$ (if the root has less than $i$ children, set $Y_{i,t} = 0$), by Lemma~\ref{lem:BP}, 
\begin{equation*}
\mathbb P(\overline{\cB}\mid \cA)\le \mathbb P\Big(\sum_{i=1}^{2\mathbb{E}[\xi_{12}]} \sum_{t\geq 1} Y_{i,t}> 2(\log n)^2 \mathbb{E}[\xi_{12}]\Big)\le 2\mathbb{E}[\xi_{12}]\mathbb P\Big(\sum_{t\geq 1} Y_{1,t}> (\log n)^2\Big) = o(n^{-3}),
\end{equation*}
\noindent
and hence $\mathbb P(\cB) \ge 1 - \mathbb P(\overline{\cB}\mid \cA) - \mathbb P(\overline{\cA}) = 1 - o(n^{-3})$.

Consider now the one-type branching process $(Z_t)_{t\geq 0}$ having offspring distribution
$$
\eta = \eta_0 + \sum_{i=1}^{2(\log n)^2 \mathbb{E}[\xi_{12}]} \eta_i,
$$
where $\eta_0$ is distributed as $\xi_{11}$ and $(\eta_i)_{i\geq 1}$ is a sequence of iid copies of $\xi_{21}$. Intuitively we may think of $(Z_t)_{t\ge 0}$ as the original branching process $(X_t)_{t\ge 0}$ where all the vertices of type $2$ have been suppressed, connecting each type $1$ vertex to its closest ancestor of type $1$. We have
$$
\mathbb{E}[\eta]\leq \frac{2}{\omega}+ 2(\log n)^2 \mathbb{E}[\xi_{12}] \mathbb{E}[\xi_{21}]\leq  \frac{3}{\omega},
$$
Denote by $T_X, T_Y,T_Z$ the total progenies of the processes $(X_t)_{t\ge 0}, (Y_t)_{t\ge 0}, (Z_t)_{t\ge 0}$, respectively. 
First, by Lemma~\ref{lem:BP}, $\mathbb{P}(T_Z\leq (\log n)^2) = 1 - o(n^{-2})$.
Now, observe that
$$
\mathbb{P}(T_X\geq (\log n)^4 \mathbb{E}[\xi_{12}]) \leq \mathbb{P}(T_Z \ge (\log n)^2)+\mathbb{P}(\sup_{1\leq i\leq n}T_{Y^{(i)}}\geq (\log n)^2 \mathbb{E}[\xi_{12}]), 
$$
where $((Y^{(i)}_t)_{t\geq 0})_{i\geq 1}$ is a sequence of iid copies of $(Y_t)_{t\geq 0}$: indeed, if none of the two events on the right hand side happens, then the total progeny of type 1 is less than $(\log n)^2$, and each such vertex has at most $(\log n)^2 \mathbb{E}[\xi_{12}]$ descendants of type 2, giving a total of less than $(\log n)^4 \mathbb{E}[\xi_{12}]$. Both events on the right hand side happen with probability $o(n^{-2})$, which finishes the proof.
\end{proof}

Suppose we start exploring from $S_0= S^{i-1}\cup\{v\}$. The number of edges between $S_0$ and $V\setminus S_0$ is at most
$$
d(v)+\sum_{j=0}^{i-1} d_j n^{\alpha_j} \sim n^{\alpha_{i-1}+\beta_{i-1}}
$$
and, by Chernoff's inequality, with probability $1-o(n^{-2})$, at most $O\left(p n^{\alpha_{i-1}+\beta_{i-1}}\right)=O\left(n^{1-\beta_i- \frac{3}{2\cdot 5^{i+1}}}\right)$ of them percolate. Now, for each edge that percolates, we start an exploration process to bound the number of vertices in the component containing that edge. Since all vertices in $S^{i-1}$ have already been explored, each individual exploration process is stochastically dominated by $(X_{t})_{t\ge 0}$, and
by Claim~\ref{claim X}, it has size $O\left((\log n)^4 n^{\frac{3}{8\cdot 5^{i+1}}}\right)$ with probability at least $1-o(n^{-2})$. By a union bound over all choices of $v$ and over all edges between $S_0$ and $V\setminus S_0$, we obtain that a.a.s.\ any component has order 
$$|S_0| + O\left( n^{1-\beta_i- \frac{3}{2\cdot 5^{i+1}}}\cdot(\log n)^4n^{\frac{3}{8\cdot 5^{i+1}}} \right) = O\left((\log n)^4 n^{1-\beta_i- \frac{9}{8\cdot 5^{i+1}}}\right),$$
which concludes the proof since $\delta_i<\frac{9}{8\cdot 5^{i+1}}$.
\end{proof}

\paragraph{Acknowledgements.}The authors are grateful to the two anonymous referees for many useful comments and suggestions.

\bibliographystyle{plain}
\bibliography{References}

\end{document}